\documentclass[psamsfonts]{amsart}
\usepackage[utf8]{inputenc}

\usepackage{amssymb,amsfonts}
\usepackage{enumerate}
\usepackage{mathrsfs}
\usepackage{url}
\usepackage{mathtools}
\usepackage[font=small,labelfont=bf]{caption}
\usepackage[all,arc]{xy}
\usepackage{xcolor}
\usepackage{scalerel}
\usepackage{stackengine,wasysym}
\usepackage{xcolor,url}
\usepackage{hyperref}
\usepackage{todonotes}
\hypersetup{
    colorlinks,
    citecolor=blue,
    filecolor=blue,
    linkcolor=black,
    urlcolor=blue
}

\newtheorem{thm}{Theorem}[section]
\newtheorem{cor}[thm]{Corollary}
\newtheorem{prop}[thm]{Proposition}

\newtheorem{lem}[thm]{Lemma}

\theoremstyle{definition}
\newtheorem{defn}[thm]{Definition}

\newtheorem{fact}[thm]{Fact}

\theoremstyle{remark}
\newtheorem{rem}[thm]{Remark}

\makeatletter
\let\c@equation\c@thm
\makeatother
\numberwithin{equation}{section}

\def\Ind{\setbox0=\hbox{$x$}\kern\wd0\hbox to 0pt{\hss$\mid$\hss} \lower.9\ht0\hbox to 0pt{\hss$\smile$\hss}\kern\wd0} 

\def\Notind{\setbox0=\hbox{$x$}\kern\wd0\hbox to 0pt{\mathchardef \nn=12854\hss$\nn$\kern1.4\wd0\hss}\hbox to 0pt{\hss$\mid$\hss}\lower.9\ht0 \hbox to 0pt{\hss$\smile$\hss}\kern\wd0} 

\def\ind{\mathop{\mathpalette\Ind{}}}

\title{Some model theory of quadratic geometries}

\thanks{The authors were supported by London Mathematical Society Research in Pairs Grant, \#42255. NR was supported by NSF grant DMS-2246992.  }

\author[C. Kestner]{Charlotte Kestner}
\address{Department of Mathematics \\
Imperial College \\ London }
\email{c.kestner@imperial.ac.uk}
\urladdr{\href{https://www.imperial.ac.uk/people/c.kestner}{https://www.imperial.ac.uk/people/c.kestner}}

\author[N. Ramsey]{Nicholas Ramsey}
\address{Department of Mathematics \\
University of Notre Dame\\
 USA}
\email{sramsey5@nd.edu}
\urladdr{\href{https://math.nd.edu/people/faculty/nicholas-ramsey/}{https://math.nd.edu/people/faculty/nicholas-ramsey/}}

\date{\today}

\begin{document}

\maketitle

\begin{abstract}
Orthogonal spaces are vector spaces together with a quadratic form whose associated bilinear form is non-degenerate.  Over fields of characteristic two, there are many quadratic forms associated to a given bilinear form and quadratic geometries are structures that encode a vector space over a field of characteristic 2 with a non-degenerate bilinear form together with a space of associated quadratic forms.  These structures over finite fields of characteristic 2 form an important part of the basic geometries that appear in the Lie coordinatizable structures of Cherlin and Hrushovski.  We (a) describe the respective model companions of the theory of orthogonal spaces and the theory of quadratic geometries and (b) classify the pseudo-finite completions of these theories. We also (c) give a neostability-theoretic classification of the model companions and these pseudo-finite completions. This is a small step towards understanding the analogue of the Cherlin-Hrushovski theory of Lie coordinatizable structures in a setting where the involved fields may be pseudo-finite.
\end{abstract}

\setcounter{tocdepth}{1}
\tableofcontents

This paper concerns the theory of pseudo-finite and generic orthogonal spaces and quadratic geometries.  An \emph{orthogonal space} is a vector space $V$ over a field $K$ together with a quadratic form $q$ such that the associated bilinear form is non-degenerate. When $K$ is of characteristic $0$ or odd characteristic, the bilinear form and the quadratic form completely determine one another; in this situation, the model theory of vector spaces with a bilinear form and the model theory of vector spaces with a quadratic form are completely equivalent. In characteristic $2$, however, there are usually many quadratic forms associated to a given bilinear form.  Indeed, there is a regular action of the dual space $V^{*}$ on the set of all such quadratic forms. A \emph{quadratic geometry}, for us, is a vector space over a field of characteristic $2$ together with a space $Q$ of quadratic forms on $V$ with a regular action of $V$ on $Q$.  These structures, over finite fields, were introduced in model theory via the work of Kantor, Liebeck, and Macpherson, where they appeared in their classification of primitive smoothly approximable $\aleph_{0}$-categorical structures.  Later, they played an important role in the Cherlin-Hrushovski theory of quasi-finite theories, where both orthogonal spaces and quadratic geometries are among the basic coordinatizing structures in the theory of Lie coordinatizable structures. 

We are interested in studying theories of orthogonal spaces and quadratic geometries without the restriction that the underlying field be finite.  A novel feature of our setting is that we are forced to deal with `extension of scalars'. This doesn't happen in the finite field setting, since the cardinality of the field, in this case, is pinned down by the theory. Over infinite fields, the field sort becomes a dynamical object, varying across the elementary class. We build on the related analysis of infinite-dimensional vector spaces over algebraically closed fields (of characteristic $\neq 2$) with symmetric or alternative bilinear forms inaugurated by Granger \cite{granger1999stability}. Following the corrections to Granger's quantifier elimination results in \cite{ChernikovHempel3} and \cite{abdaldaim2023higher}, we prove that orthogonal spaces and quadratic geometries eliminate quantifiers in a reasonable language, down to quantiifers over the field. Along the way, for both the theories of orthogonal spaces and quadratic geometries, we classify the pseudo-finite completions and axiomatize their model companions. For orthogonal spaces, we show that the model companion is exactly the theory of an infinite dimensional vector spaces over an algebraically closed field of characteristic $2$ equipped with a quadratic form whose associated bilinear form is non-degenerate. For quadratic geometries, however, the model companion depends significantly on whether or not the Witt defect is present in the language. In both cases, the model companion includes axioms asserting that the vector space is infinite dimensional and the bilinear form associated to each quadratic form is non-degenerate.  Without the Witt defect, the underlying field is algebraically closed, but, in the presence of the Witt defect, something interesting happens: the underlying field is pseudo-algebraically closed with absolute Galois group $\mathbb{Z}_{2}$. 

By analogy to the theory of smoothly approximable homogeneous structures, we prove that the orthogonal spaces and quadratic geometries we consider are homogeneous and smoothly approximable \emph{relative to the field}. When the underlying field is infinite, we prove `smooth approximability' by finite dimensional subspaces over the same field. When the underlying field is algebraically closed, infinite dimensional orthogonal and quadratic spaces will not be pseudo-finite, but our results entail that they are nontheless `pseudo-finite dimensional.' Consequently, these structures provide a stock of interesting examples to consider in relation to the not-yet-existent theory of smoothly approximable structures over infinite fields, the possibility of which was raised by Hrushovski \cite{hrushovski1993finite}.

Finally, we apply our quantifier elimination results in Section \ref{section:neostability} to give an analysis of the placement of these structures on the neostability-theoretic map. Although the presence of pseudo-finite fields (or, more generally, non-separably closed PAC fields) is incompatible with $n$-dependence for any $n$ \cite{hempel2016n}, we prove that orthogonal spaces and quadratic geometries over algebraically closed fields are NFOP$_{2}$.  The property NFOP$_{2}$ is a ternary notion of stability introduced recently by Terry and Wolf \cite{terry2022irregular}.  Our result mirrors results from \cite{abdaldaim2023higher} which established that both smoothly approximable structures as well as Granger's bilinear forms over algebraically closed fields are NFOP$_{2}$.  The proof for orthogonal geometries turns out to be especially interesting, due to the complexity of terms interrelating the three distinct sorts. We also show that both the pseudo-finite structures and the model companions are NSOP$_{1}$ by giving an explicit characterization of Kim-independence.  

\section{Orthogonal Spaces}  

\subsection{The basic theory}  The base language $L_{V,K}$ will have two sorts $V$ and $K$.  On $K$, there is the language of rings and, on $V$, there is the language of abelian groups, and there is a function $\cdot : K \times V \to V$ for scalar multiplication. Towards quantifier-elimination, we also include in $L_{V,K}$ linear independence predicates and coordinate functions.  More precisely, for each $n \geq 1$, we include in $L_{V,K}$ an $n$-ary predicate $\theta_{n}$ on $V$ and, for each $i = 1, \ldots, n$, an $n+1$-ary function $\pi_{n,i} : V^{n+1}\to K$.  The theory $T^{0}_{V,K}$ asserts that $K$ is a field and $V$ is a $K$-vector space with respect to the scalar multiplication $\cdot$. The theory $T^{0}_{V,K}$ also contains axioms asserting that $\theta_{n}(v_{1}, \ldots, v_{n})$ holds if and only if $v_{1}, \ldots, v_{n}$ are linearly independent.  Finally, there are axioms asserting that if $v_{1}, \ldots, v_{n} \in V$ are linearly independent and $w = \sum_{i = 1}^{n} \alpha_{i}v_{i}$, then $\pi_{n,i}(v_{1},\ldots, v_{n},w) = \alpha_{i}$ (and is zero otherwise). 

The language $L_{O}$ extends $L_{V,K}$ with the addition of a unary function symbol $q : V \to K$.  The theory $T^{0}_{O}$ extends $T^{0}_{V,K}$ by asserting that $q$ is a quadratic form on $V$, i.e. $q(\alpha v) = \alpha^{2} q(v)$ for all $\alpha \in K$ and $v \in V$, and the function $\beta : V^{2} \to K$ defined by $\beta(x,y) = q(x+y) - q(x) - q(y)$ is a non-degenerate symmetric $K$-bilinear map on $V$. We will also add an axiom to $T^{0}_{O}$ asserting that $K$ has characteristic $2$. Models of $T^{0}_{O}$ will be referred to as \emph{orthogonal spaces}.  

We note that the requirement that $K$ has characteristic $2$ is not necessary for the study of orthogonal spaces, but, in odd characteristic, the quadratic form is completely determined by the associated bilinear form and so the study of orthogonal spaces reduces entirely to that of bilinear forms undertaken by Granger \cite{granger1999stability}.  

Our goals are (1) to axiomatize all pseudo-finite completions of $T^{0}_{O}$ and also to axiomatize the model companion of $T^{0}_{O}$.  Since we will need to talk about the theory of the associated fields, we will write ACF$_{2}$ to denote the theory of algebraically closed fields of characteristic $2$ and PSF$_{2}$ to denote the theory of pseudo-finite fields of characteristic $2$. 

\begin{rem}
    Note that that when $K$ is finite, the linear independence predicates and coordinate functions are not needed; one can instead name the elements of $K$ with constants to obtain quantifier elimination, see \cite[Lemma 2.2.8]{cherlin2003finite}.  For the closely related two-sorted theory of infinite dimensional vector spaces over an algebraically closed field of odd characteristic with a symmetric or alternating bilinear form, quantifier elimination in a language without the coordinate functions was claimed in \cite{granger1999stability}, but an example showing that they are necessary in the case that the field is infinite appears in \cite[Section 2.2]{dobrowolskisets}.  Quantifier elimination for the two-sorted theory of infinite dimensional vector spaces with a non-degenerate alternating or symmetric bilinear form, in a language with coordinate functions, was ultimately proved independently in \cite{abdaldaim2023higher} and \cite{ChernikovHempel3}.    
\end{rem}

\subsection{Extension of scalars}
In this subsection, we explain how to extend the quadratic form as the field grows. 

\begin{defn}
    Suppose $M = (V,K,q)$ is an orthogonal space and $K'$ is a field extension of $K$.  We define the \emph{extension of scalars} of $M$ to $K'$ to be the structure $M' = (K' \otimes_{K} V, K', q')$ where the quadratic form $q'$ is defined as follows: 
    $$
q'\left( \sum_{i = 1}^{k} \alpha_{i} \otimes v_{i} \right) = \sum_{i = 1}^{k} \alpha_{i}^{2} q(v_{i}) + \sum_{i < j} \alpha_{i}\alpha_{j} \beta(v_{i},v_{j}).
    $$
\end{defn}

It is well-known that extension of scalars, as defined above, produces a quadratic form and that the bilinear form associated to $q'$ is the same as the natural extension of $\beta$ to a $K'$-bilinear map of $K' \otimes_{K} V$, which is non-degenerate if and only if $\beta$ is.  See, e.g., \cite[Ch. II, Section 7]{elman2008algebraic}. This easily yields the following: 

\begin{lem}
    Suppose $M = (V,K,q)$ is an orthogonal space and $K'$ is a field extension of $K$.  If $M' = (V\otimes_{K} K', K',q')$ is the extension of scalars of $M$ to $K'$, then $M'$ is an orthogonal space and the map $M \to M'$, which is the inclusion on the field sort and the map $v \mapsto 1 \otimes v$ on the vector space sort, is an embedding of $L_{O}$-structures. 
\end{lem}

%

Towards axiomatizing a model companion of $T^{0}_{O}$, we note the following consequence:

\begin{cor}
    If $M = (V,K,q)$ is an existentially closed model of $T^{0}_{O}$, then $K$ is algebraically closed. 
\end{cor}

\begin{lem}\label{lem:substructure}
 Suppose $M = (V,K,q)$ is a substructure of $N \models T^{0}_{O}$ and $K'$ is a field with $K \subseteq K' \subseteq K(N)$.  Then the substructure of $N$ generated by $M$ and $K'$ in $N$ satisfies
 $$
\langle M,K' \rangle \cong (K' \otimes_{K} V, K',q').
 $$
 where $q'$ is the canonical extension of $q$ to $K' \otimes_{K} V$.
\end{lem}

\begin{proof}
 Note that because $M$ is a substructure of $N$ and we have the linear independence predicates in the language, vectors in $V$ that are linearly independent in $M$ remain linearly independent in $N$. Hence, we may identify $K' \otimes_{K} V$ with all elements of $V(N)$ that can be expressed as $\sum_{i = 1}^{n} \alpha_{i} \cdot v_{i}$ with $\alpha_{i} \in K'$ and $v_{i} \in V$. Having made this identification, it is clear that the underlying set of $(K' \otimes_{K} V, K', q')$ is a subset of the structure $\langle M,K' \rangle$ so we merely need to show that this subset is a substructure, i.e. that it is closed under the functions of $L_{O}$. Closure under the field operations on the field sort, the abelian group operations on the vector space sort, and by scalar multiplication are totally clear, so we are left with checking that this set is closed under the quadratic form and the coordinate functions. For the quadratic form, this follows from the formula given above:
$$
q'\left( \sum_{i = 1}^{n} \alpha_{i}\cdot v_{i}\right) = \sum_{i = 1}^{n} \alpha_{i}^{2}q(v_{i}) + \sum_{i < j} \alpha_{i}\alpha_{j} \beta(v_{i},v_{j})
$$
and since for all $i,j$, we have $q(v_{i}), \beta(v_{i},v_{j}) \in K$ and $\alpha_{i} \in K'$, we know that the right-hand side describes an element in $K'$. Now for the coordinate functions, we suppose we are given $w_{1}, \ldots, w_{n+1} \in K' \otimes_{K} V$ such that $w_{1}, \ldots, w_{n}$ are $K'$-linearly independent and $w_{n+1}$ is in their span (as computed in $N$).  We want to show that $w_{n+1}$ is in their $K'$-span. This is easy but we spell out the details.

We can find $n$ linearly independent vectors $v_{1}, \ldots, v_{n} \in V$ such that each $w_{i}$ can be written 
$$
w_{i} = \sum_{j = 1}^{n} \alpha_{ij} v_{j}.
$$
If we have $w_{n+1} = \sum_{i = 1}^{n} \lambda_{i} w_{i}$, then, since we also have 
$$
w_{n+1} = \sum_{j = 1}^{n} \alpha_{n+1,j} v_{j},
$$
we get that 
$$
\left[ 
\begin{matrix}
    \alpha_{1,1} & \ldots & \alpha_{1,n} \\
    \vdots & \ddots & \\
    \alpha_{n,1} & \dots & \alpha_{n,n} 
\end{matrix}
\right] \left[
\begin{matrix}
    \lambda_{1}  \\
    \vdots \\
    \lambda_{n}
\end{matrix}
\right] = \left[
\begin{matrix}
    \alpha_{n+1,1}  \\
    \vdots \\
    \alpha_{n+1,n}
\end{matrix}
\right].
$$
Moreover, fixing coordinates so that $v_{i} = e_{i}$, the matrix $B = (\alpha_{i,j}) \in \mathrm{GL}_{n}(K')$ since it takes the standard basis to $w_{1}, \ldots, w_{n}$. Hence we have
$$
B^{-1} \left[
\begin{matrix}
    \alpha_{n+1,1}  \\
    \vdots \\
    \alpha_{n+1,n}
\end{matrix}
\right] = \left[
\begin{matrix}
    \lambda_{1}  \\
    \vdots \\
    \lambda_{n}
\end{matrix}
\right].
$$
This shows that the $\lambda_{i}$ are obtained by applying a matrix in $\mathrm{GL}_{n}(K')$ to a vector in $(K')^{n}$ so each $\lambda_{i} \in K'$.  This shows that, for each $1 \leq i \leq n$
$$
\pi^{N}_{n,i}(w_{1}, \ldots, w_{n}, w_{n+1}) = \lambda_{i} \in K',
$$
so we have established that $(K' \otimes_{K} V,K',q')$ is equal to the substructure of $N$ generated by $K'$ and $M$. 
\end{proof}

\begin{defn}
    Suppose $A = (V(A),K(A),q)$ is an orthogonal space. 
    \begin{enumerate}
        \item We say that a vector $v \in V(A)$ is \emph{singular} if $q(v) = 0$.  
        \item  We say that $(V(A),q)$ is a \emph{hyperbolic plane} if $\mathrm{dim}(V(A)) = 2$ and there is a basis $\{u,v\}$ of $V(A)$ consisting of singular vectors with $\beta(u,v) = 1$. Such a basis is called a \emph{hyperbolic pair}.
        \item We say that $(V(A),q)$ is \emph{hyperbolic} if $V(A)$ can be written as an orthogonal direct sum of hyperbolic planes.
        \item If $V(A)$ has no non-zero singular vectors, it is called \emph{definite}.  
    \end{enumerate}   We may abuse terminology and refer to two sorted structure $A$ as a hyperbolic plane/a hyperbolic space when we mean that $V(A)$ is.
\end{defn}

We will need the following basic linear algebra facts: 

\begin{fact} \label{fact:Aschbacher}
Let $(V,\beta)$ be non-degenerate and $U$ a subspace of $V$.
    \begin{enumerate}
        \item $\mathrm{dim}(U^{\perp}) = \mathrm{codim}(U)$. \cite[19.2]{aschbacher2000finite}
        \item $(U^{\perp})^{\perp} = U$. \cite[19.3(2)]{aschbacher2000finite}
        \item The subspace $U$ is non-degenerate if and only if $V = U \oplus U^{\perp}$. \cite[19.3(1)]{aschbacher2000finite}
        \item Let $x \in V$ be a non-zero singular vector and suppose $y \in V - x^{\perp}$.  Then the span of $x$ and $y$ is a hyperbolic plane and $x$ is contained in a hyperbolic pair of $\langle x,y\rangle$. \cite[19.2]{aschbacher2000finite}
        \item If $(V,q)$ is an orthogonal space over $F$ with associated bilinear form $\beta$ and $\mathrm{char}(F) = 2$, then $V$ is of even dimension. \cite[19.17]{aschbacher2000finite} 
    \end{enumerate}
\end{fact}

If $(V,q)$ and $(V',q')$ are both vector spaces over a field $K$ equipped with bilinear forms, an \emph{isometry} from $(V,q)$ to $(V',q')$ is a $K$-linear isomorphism $f: V\to V'$ such that $q(v) = q'(f(v))$ for all $v \in V$. The following fact is known as \emph{Witt's Lemma}:
\begin{fact} \label{Witt} \cite[Ch. 20]{aschbacher2000finite}
    Suppose $(V,q)$ is an orthogonal space over $K$ and we are given an isometry $f: U \to U'$ where $U$ and $U'$ are subspaces of $V$.  Then there is some isometry $\tilde{f} : V \to V$ extending $f$. 
\end{fact}

\begin{rem}
    We will often use Fact \ref{Witt} for the $2$-sorted structures encoding orthogonal spaces in the following way:  given an orthogonal space $M = (K,V,q)$ and a $K$-isomorphism $f$ between substructures $A = (K,U,q|_{U})$ and $B = (K,U',q|_{U'})$, there is some $\sigma \in \mathrm{Aut}(M/K)$ extending $f$. 
\end{rem}

For a field $K$ of characteristic $2$, the \emph{Artin-Schreier map} is defined by $\wp(x) = x^{2} - x$.  If $F/K$ is a quadratic extension, then $F$ is given by adjoining a root to a polynomial of the form $X^{2} - X - b$ for some $b \in K$.  Such polynomials are called \emph{Artin-Schreier polynomials}.  

Suppose $P(X) = X^{2} - X - b$ is an irreducible Artin-Schreier polynomial over $K$, $\alpha$ is a root of $P(X)$ in the splitting field $F$ of $P$, and $\sigma \in \mathrm{Gal}(F/K)$ is a generator.  Then over $F$, the polynomial factors into 
$$
P(X) = (X - \alpha)(x - \sigma(\alpha)),
$$
which yields, by comparison of coefficients, that $\alpha + \sigma(\alpha) = 1$ and $\alpha \sigma(\alpha) = b$. 

Suppose $Q(X) = X^{2} - X - b'$ is an Artin-Schreier polynomial over $K$ that gives the same extension\textemdash i.e. adding a root $\alpha'$ results in an extension with $F = K(\alpha) = K(\alpha')$.  In this case, it follows that $\alpha' = c \alpha + d$ for some $c,d \in K$.  Then we have 
\begin{eqnarray*}
b' &=& \alpha'\sigma(\alpha')\\
&=& (c + d \alpha)(c + d \sigma(\alpha)) \\
&=& c^{2} + cd\alpha + cd \sigma(\alpha) + d^{2} \alpha \sigma(\alpha) \\
&=& c^{2} + cd + d^{2}b,
\end{eqnarray*}
so $b' = c^{2} +cd + d^{2}b$.  This will turn out to be useful in Lemma \ref{lem:same_quad_ext} below.  

\begin{lem}\label{lem:same_quad_ext}
    Suppose we are given that $(V_{0},q_{0})$ and $(V_{1},q_{1})$ are both $2$-dimensional orthogonal spaces over a perfect field $K$ of characteristic $2$. Suppose $x,y$ is a basis of $V_{0}$ with $q_{0}(x) = 1$, $q_{0}(y) = b_{0}$, and $\beta_{0}(x,y) = 1$ and, likewise, $u,v$ is a basis of $V_{1}$ with $q_{1}(u) = 1$, $q_{1}(v) = b_{1}$, and $\beta_{1}(u,v) = 1$.  Then the Artin-Schreier polynomials $t^{2} + t + b_{0}$ and $t^{2} + t + b_{1}$ give rise to the same quadratic extension of $K$ if and only if $b_{0} - b_{1} \in \wp(K)$.
\end{lem}

\begin{proof}
   From left to right, we let $F$ denote the common splitting field of $t^{2} + t + b_{0}$ and $t^{2} + t + b_{1}$ and let $\alpha_{0}$ and $\alpha_{1}$ denote roots of these respective polynomials. So $F = K(\alpha_{0}) = K(\alpha_{1})$ and we may write $\alpha_{1} = c + d\alpha_{0}$ for some $c,d \in K$.  Let $\sigma \in \mathrm{Gal}(F/K)$ be a generator. Recall that the norm map $N^{F}_{K} : F \to K$ is defined by $N^{F}_{K}(z) = z \sigma(z)$.  
    
    We first will argue that $(V_{i},q_{i}) \cong (F,N^{F}_{K})$ via the isomorphism $x \mapsto 1, y \mapsto \alpha_{0}$ for $i = 0$ and via the isomorphism $u \mapsto 1$, $v \mapsto \alpha_{1}$ in the case that $i = 1$, where we view $F$ as a $2$-dimensional $K$-vector space.  Notice that $q_{0}(x) = N^{F}_{K}(1) = 1$, $q_{0}(y) = N^{F}_{K}(\alpha_{0}) = \alpha_{0} \sigma(\alpha_{0}) = b_{0}$ and, moreover, 
    $$
    N^{F}_{K}(1+\alpha_{0}) = (1 + \alpha_{0})(1 + \sigma(\alpha_{0})) = 1 + \alpha_{0} + \sigma(\alpha_{0}) + \alpha_{0} \sigma(\alpha_{0}) = 1+1+b_{0} = b_{0},
    $$
    which entails that if $\beta_{F}$ is the bilinear form associated to $(F,N^{F}_{K})$, then we have 
    \begin{eqnarray*}
    \beta_{F}(1,\alpha_{0}) &=& N^{F}_{K}(1 + \alpha_{0}) + N^{F}_{K}(1) + N^{F}_{K}(\alpha_{0}) \\
    &=& \beta_{0}(x,y).
    \end{eqnarray*}
    Since we have shown $q_{0}(x) = N^{F}_{K}(1)$, $q_{0}(y) = N^{F}_{K}(\alpha_{0})$, and $\beta_{0}(x,y) = \beta_{F}(1,\alpha_{0})$, it follows that the map $x \mapsto 1$, $y \mapsto \alpha_{0}$ extends to an isomorphism $(V_{0},q_{0}) \to (F,N^{F}_{K})$.  An identical argument shows that the map $u \mapsto 1$, $v \mapsto \alpha_{1}$ extends to an isomorphism $(V_{1},q_{1}) \to (F,N^{F}_{K})$. 

    Now we can calculate, using that $\beta_{1}(u,v) = 1$:
    \begin{eqnarray*}
        1 &=& N^{F}_{K} (1 + \alpha_{1}) + N^{F}_{K}(1) + N^{F}_{K}(\alpha_{1}) \\
        &=& N^{F}_{K}(1 + c + d\alpha_{0}) + 1 + N^{F}_{K}(c + d \alpha_{0}) \\
        &=& (1 + c + d \alpha_{0})(1 + c + d\sigma(\alpha_{0})) + 1 + (c + d \alpha_{0})(c + d \sigma(\alpha_{0})) \\
        &=& (1 + c)^{2} + (1+c)d(\alpha_{0} + \sigma(\alpha_{0})) + d^{2}\alpha_{0}\sigma(\alpha_{0}) + 1 + (c + d \alpha_{0})(c + d \sigma(\alpha_{0})) \\
        &=& 1 + c^{2} + d + cd + b_{0}d^{2} + 1 + c^{2} + cd(\alpha_{0} + \sigma(\alpha_{0})) + d^{2}\alpha_{0}\sigma(\alpha_{0}) \\
        &=& d + cd + b_{0}d^{2} + cd + d^{2} b_{0} \\\\
        &=& d.
    \end{eqnarray*}
    This shows $d = 1$, so $\alpha_{1} = c + \alpha_{0}$.  Next, we observe 
    \begin{eqnarray*}
        b_{1} &=& N^{F}_{K}(\alpha_{1}) \\
        &=& N^{F}_{K}(c + \alpha_{0}) \\
        &=& (c + \alpha_{0}) (c + \sigma(\alpha_{0})) \\
        &=& c^{2} + c(\alpha_{0} + \sigma(\alpha_{0})) + \alpha_{0}\sigma(\alpha_{0}) \\
        &=& c^{2} + c + b_{0},
    \end{eqnarray*}
    so $b_{1} - b_{0} = c^{2} + c \in \wp(K)$.  This completes the proof.

    For the other direction, we let $b_{0} - b_{1} \in \wp(K)$, then $t^{2} + t + b_{0}$ and $t^{2} + t + b_{1}$ give rise to the same Artin-Schreier extension, without requiring the use of any further assumptions. 
\end{proof}

\begin{lem} \label{lem:hypext}
    Suppose $M \models T^{0}_{O}$ and $A \subseteq M$ is a finitely generated substructure such that $\beta$ is non-degenerate on $V(A)$ and $\mathrm{dim}(V(A)) = 2n$. 
    \begin{enumerate}
        \item If $K(M) \models \mathrm{ACF}_{2}$, then there is a finitely generated $A'$ with $A \subseteq A' \subseteq M$ such that $\mathrm{dim}(V(A')) = 2n$ (i.e. $A'$ is an extension by scalars of $A$) and $(V(A'),q)$ is a hyperbolic space over $K(A')$. 
        \item If $K(M) \models \mathrm{PSF}_{2}$, then there is a finitely generated $A'$ with $A \subseteq A' \subseteq M$ such that $\mathrm{dim}(V(A')) = 2n$ and $(V(A'),q)$ is an orthogonal space over $K(A')$ which is either hyperbolic or is the orthogonal direct sum of a hyperbolic space and a 2-dimensional definite orthogonal space that contains a vector with $q(v) = 1$. 
    \end{enumerate}
\end{lem}

\begin{proof}
    We first will argue by induction on $n$ that in both cases (1) and (2) that there is such an $A' \supseteq A$ with $(V(A'),q)$ containing a hyperbolic subspace of dimension $2n - 2$. For $n = 1$ there's nothing to show. 

    Assume the conclusion has been established for $n$ and $(V(A),q)$ is non-degenerate of dimension $2n+2$.  After replacing $A$ by a finitely generated extension, we may assume, by the inductive hypothesis, that $V(A)$ contains a hyperbolic subspace $H_{0} \subseteq V(A)$ of dimension $2n - 2$.  Since $H_{0}$ is non-degenerate, we have, by Fact \ref{fact:Aschbacher}(3), $V(A) = H_{0} \oplus H_{0}^{\perp}$ and $H_{0}^{\perp}$ is non-degenerate with $\dim(H_{0}^{\perp}) = \mathrm{codim}(H_{0}) = 4$.  Replacing $V(A)$ with $H_{0}^{\perp}$, we are reduced to showing that if $V(A)$ has dimension $4$, then there is a finitely generated extension of scalars which contains a hyperbolic plane. 

    So now we assume $V(A)$ is non-degenerate and has dimension $4$.  Pick any non-zero $x \in V(A)$.  Assuming $x$ is non-singular, we know that, since $K(M)$ is perfect, $\sqrt{q(x)} \in K(M)$, so, replacing $A$ by a finitely generated extension, we may assume $\sqrt{q(x)} \in K(A)$.  Then replacing $x$ by $\frac{x}{\sqrt{q(x)}}$, we may assume $q(x) = 1$.  Then, by Fact \ref{fact:Aschbacher}(1), $\dim(x^{\perp}) = \mathrm{codim}(\langle x \rangle) = 3$ so we can pick some $y \in x^{\perp} - \langle x \rangle$ (where both terms are computed in $V(A)$. Again, assuming $y$ is non-singular, after maybe passing to a finitely generated extension, we may assume $\sqrt{q(y)} \in K(A)$, so that, after scaling, we have $q(y) = 1$. Then, since $x$ and $y$ are orthogonal, we have $q(x+y) = q(x) + q(y) = 2q(x) = 0$.  Hence, we have found a singular vector and, thus, $V(A)$ contains a hyperbolic plane by Fact \ref{fact:Aschbacher}(4). 

    So now we prove (1) and (2) of the statement.  Suppose $V(A)$ has dimension $2n$.  We have shown that there is a finitely generated extension $A'$ of $A$ and a $2n-2$ dimensional hyperbolic subspace $H \subseteq V(A')$.  Then, working in $V(A')$, we have $\mathrm{dim}(H^{\perp}) = \mathrm{codim}(H) = 2$.  Moreover, $H \cap H^{\perp} = 0$, so we know that $H^{\perp}$ is non-degenerate by Fact \ref{fact:Aschbacher}(3) and $V(A')$ is an orthogonal direct sum of $H$ and $H^{\perp}$.  If $H^{\perp}$ contains a singular vector, then we are done and otherwise $H^{\perp}$ is definite, so we have already proved (2).  To prove (1), we just note that, in the case that $K(M) \models ACF_{2}$, we can, after passing to a finitely generated extension, ensure that there is a singular vector in $H^{\perp}$.  Suppose that there's no singular vector in $H^{\perp}$.  Then, we have seen that, after passing to a finitely generated extension, we can assume there is some $v \in H^{\perp}$ with $q(v) = 1$.  By non-degeneracy, there must be some $y \in H^{\perp}$ with $\beta(x,y) \neq 0$ since, and, after scaling, we may assume $\beta(x,y) = 1$. Then, denoting $q(y)$ by $b$, we consider the polynomial 
    $$
    P(t) = q(tx + y) = \beta(tx,y) + q(tx) + q(y) = t^{2} + t + b.
    $$
    We know that $P(t)$ is not irreducible over $K(M)$, so we may pass to a finitely generated extension $A''$ of $A'$ in which $K(A'')$ contains a root $\alpha$ of $P(t)$.  Then we see $\alpha x + y \in V(A'')$ and, in fact, in $H^{\perp}$, computed in $V(A'')$.  Therefore, we have found a singular vector in $H^{\perp}$ in $V(A'')$ so $V(A'')$ can be written as an orthogonal direct sum of $n$ hyperbolic planes. 
\end{proof}

\begin{lem} \label{lem:hyperbolic}
    Suppose $M \models T^{0}_{O}$ with $V(M)$ infinite dimensional and suppose further that $A \subseteq M$ is a finitely generated substructure such that $\beta$ is non-degenerate on $V(A)$ and $\mathrm{dim}(V(A)) = 2n$. If $K(M)$ is a perfect field with a unique quadratic Galois extension, then there is a finitely generated $A' \supseteq A$ such that $\mathrm{dim}(V(A')) = 2n+2$ and $(V(A'),q)$ is hyperbolic. 
\end{lem}

\begin{proof}
    By Lemma \ref{lem:hypext}, we may assume, after replacing $A$ with a finitely generated extension, that $V(A)$ is either hyperbolic or the orthogonal direct sum of $n-1$ hyperbolic planes and a non-degenerate definite space which contains a vector $x$ with $q(x) = 1$. The hyperbolic case is easy to handle using Lemma \ref{lem:hypext}:  choose any $4$-dimensional non-degenerate subspace of $V(A)^{\perp}$ in $V(M)$ (which exists since $V(M)$ is non-degenerate and infinite dimensional) and then the argument of the lemma proves that there is a hyperbolic plane in this four-dimensional space.  The structure generated by a basis of this hyperbolic plane and $V(A)$, then, gives the desired hyperbolic space of dimension $2n+2$. 

    So now we assume that $V(A)$ is the orthogonal direct sum of $n-1$ hyperbolic planes and a non-degenerate definite space, which has basis $x,y$ with $q(x) = 1$.  By non-degeneracy, we know $\beta(x,y) \neq 0$ so, after scaling $y$, we may assume $\beta(x,y) = 1$. Let $b$ denote $q(y)$.  We define a polynomial $P(t)$ by 
    $$
    P(t) = q(tx+y) = q(tx) + \beta(tx,y) + q(y) = t^{2} + t + b.
    $$
    Our assumption that $x,y$ spans a definite subspace of $V(A)$ entails that $P(t)$ is irreducible over $K(A)$.  If it is not irreducible over $K(M)$, then we may obtain a finitely generated extension $A'$ of $A$ in $K(M)$ by adjoining one of the roots $\alpha$ in $K(M)$.  Then $\alpha x + y \in V(A')$ and this is a singular vector, so, by \ref{fact:Aschbacher}(4), we get that $V(A')$ is hyperbolic and we reduce to the case handled above.  So we may assume $P(t)$ is irreducible over $K(M)$. 

    By non-degeneracy, we may choose some $u \in V(A)^{\perp}$ in $V(M)$ and, after possibly adding $\sqrt{q(u)}$ to $A$, we may assume $q(u) = 1$. Again by non-degeneracy, we may choose some $v \in V(M)$ which is in $V(A)^{\perp} - u^{\perp}$ which, after scaling, we may assume satisfies $\beta(u,v) = 1$.  Set $q(v) = b'$.  Then we define the polynomial $Q(t)$ by 
    $$
    Q(t) = q(tu + v) = t^{2} + t + b'.
    $$
    As described above, if $Q(t)$ is irreducible over $K(M)$, then there is a finitely generated extension in which we adjoin a root and hence obtain a singular vector in the span of $u$ and $v$.  This would imply that $u$ and $v$ span a hyperbolic plane by Fact \ref{fact:Aschbacher}(4), so the finitely generated extension obtained by then adding $u$ and $v$ gives us what we want.  So we may assume $Q(t)$ is also irreducible. We know that, since $P(t)$ and $Q(t)$ give rise to the same quadratic extension of $K(M)$, so, by Lemma \label{lem:same_quad_ext}, there must be some $c \in K(M)$ such that 
    $$
    b' = c^{2} + c+ b.
    $$
    By possibly extending scalars, we may assume $c \in K(A)$.  Let $v' = cu + v$.  Then we have $q(cu+v) = b' + c^{2} + c = b$ and also $\beta(u,v') = 1$.  

    Now consider the subspace spanned by $x+u$ and $y+v'$.  Since $x$ and $u$ are orthogonal, we have $q(x+u) = q(x) + q(u) = 1+1 = 0$ and, likewise, $q(y+v') = b+b = 0$. Moreover, since $x,y$ are orthogonal to $u,v$, we have 
    $$
    \beta(x+u,y+v') = \beta(x,y) + \beta(u,v') = 1+1 = 0.  
    $$
    It follows, then, that the span of $x+u$ and $y+v$ is a totally singular subspace of dimension $2$, so the span of $x,y,u,v$ is hyperbolic by \ref{fact:Aschbacher}(4).  
\end{proof}

\subsection{Back and forth for orthogonal spaces}

By a \emph{pseudo-finite} orthogonal space, we mean an infinite model $M \models T^{0}_{O}$ such that if $\phi \in L_{O}$ is a sentence with $M \models \phi$, then there is a finite $M' \models T^{0}_{O}$ with $M' \models \phi$.  Equivalently, $M$ is infinite and elementarily equivalent to an ultraproduct of finite orthogonal spaces.  The requirement that $M$ be infinite forces that either the field is an infinite pseudo-finite field or $V$ is of infinite dimension.  We are interested in describing all the possibilities for $\mathrm{Th}(M)$.

\begin{prop} \label{prop:orthogonal bnf}
    Suppose $M, N \models T^{0}_{O}$ are orthogonal spaces with $K(M),K(N)$ $\aleph_{0}$-saturated and either both models of ACF$_{2}$ or elementarily equivalent pseudo-finite fields of characteristic $2$ (as $L_{\mathrm{ring}}$-structures) and with both $V(M)$ and $V(N)$ infinite dimensional.  Then the set of isomorphisms $f:A \to B$ with $A$ and $B$ between finitely generated substructures of $M$ and $N$ such that the induced map $f|_{K(A)}$ is partial elementary (as an $L_{\mathrm{ring}}$-embedding) has the back and forth property. 
\end{prop}

\begin{proof}
    Suppose we are given an isomorphism $f: A \to B$ as in the statement, where $A \subseteq M$ and $B \subseteq N$ are finitely generated substructures.  We can write $f = (f_{K},f_{V})$ where $f_{K} : K(A) \to K(B)$ is an isomorphism of fields (which is additionally assumed to be partial elementary with respect to $\mathrm{Th}_{L_{\mathrm{ring}}}(K(M))$ in the case that $K(M)$ is pseudo-finite) and $f_{V} : V(A) \to V(B)$ is the restriction of $f$ to the vector space sort.  Fix some $c \in M \setminus A$ and we will show that we can extend $f$ to some $f'$ whose domain is $\langle A c \rangle$. We must consider three cases:

    \textbf{Case 1:  }$c \in K(M)$.

    In this case, we have, by Lemma \ref{lem:substructure}, $\langle A,c \rangle \cong (K(A)(c), K(A)(c) \otimes_{K(A)} V(A) \rangle$.  Since $B$ is finitely generated and $K(N)$ is $\aleph_{0}$-saturated, we can find some $d \in K(N)$ such that $K(A)c \equiv K(B)d$ in $K(N)$. This allows us to extend $f_{K}$ to $f'_{K} : K(A)(c) \to K(B)(d)$ defined by mapping $c \mapsto d$. Then we can define a map $f'_{V} : K(A)(c) \otimes_{K(A)} V(A) \to K(B)(d) \otimes_{K(B)} V(B)$ by setting 
    $$
    f'_{V}\left(\sum_{i = 1}^{n}\alpha_{i} \otimes v_{i}\right) = \sum_{i = 1}^{n} f'_{K}(\alpha_{i}) \otimes f_{V}(v_{i})
    $$
    whenever $\alpha_{i} \in K(A)(c)$ and $v_{i} \in V$. This gives the desired isomorphism.

    \textbf{Case 2:  }$c \in \mathrm{Span}_{V(M)}(V(A))$.  

    In this case, we have $c = \sum_{i = 1}^{n} \alpha_{i} v_{i}$ where each $\alpha_{i} \in K(M)$. Iteratively applying Case 1, we may extend $f$ to some $f'$ whose domain is $\langle A, \alpha_{1}, \ldots, \alpha_{n} \rangle$.  Since $c \in \langle A, \alpha_{1}, \ldots, \alpha_{n} \rangle$, this suffices. 

    \textbf{Case 3:  }$c \in V(M) - \mathrm{Span}_{V(M)}(V(A))$.  

    Let $v_{0}, \ldots, v_{k-1}$ be a basis for $V(A)$ and let $w_{i} = f_{V}(v_{i})$. So $w_{0}, \ldots, w_{k-1}$ is also a basis of $V(B)$. By Lemma \ref{lem:hyperbolic}, we can extend $v_{0}, \ldots, v_{k-1}$ to a $K(M)$-basis $v_{0}, \ldots, v_{2n-1}$ of a hyperbolic subspace of $V(M)$ which contains $c$ in its span. Likewise, we can extend $w_{0}, \ldots, w_{k-1}$ to a $K(M)$-basis $w_{0}, \ldots, w_{2n-1}$ of a hyperbolic subspace of $V(M)$. 
    
    Let $(v'_{i})_{i <2n}$ and $(w'_{i})_{i < 2n}$ be hyperbolic bases of $\mathrm{Span}_{K(M)}\{v_{i} : i < 2n\}$ and $\mathrm{Span}_{K(M)}\{w_{i} : i < 2n\}$ respectively.  That is, we ask that these are bases satisfying the following conditions:
    \begin{itemize}
        \item $q(v'_{i}) = q(w'_{i}) = 0$ for all $i < 2n$.
        \item For all $k,\ell < 2n$, 
        $$
        \beta(v'_{k},v'_{\ell}) = \beta(w'_{k},w'_{\ell}) = \left\{ \begin{matrix}
            1 & \text{ if } \{k,\ell\} = \{2i,2i+1\} \text{ for some }i < n \\
            0 & \text{ otherwise}
        \end{matrix}
        \right. 
        $$
    \end{itemize}
   We know for each $i < 2n$, we have 
   $$
   v'_{i} = \sum_{j  < 2n} \alpha_{i,j} v_{j}
   $$
   and, likewise, 
   $$
   w'_{i} = \sum_{j < 2n} \gamma_{i,j} w_{j},
   $$
   with each $\alpha_{i,j}, \gamma_{i,j} \in K(M)$. By Case 1, we may assume that, in fact, all $\alpha_{i,j} \in K(A)$ and all $\gamma_{i,j} \in K(B)$.  It follows, then, that 
   $$
   \mathrm{Span}_{K(A)}\{v_{i} : i < 2n\} = \mathrm{Span}_{K(A)}\{v'_{i} : i < 2n\}
   $$
   and likewise for the $w_{i}$ and $w'_{i}$.  Define $C = \langle K(A), \{v'_{i} : i < 2n\} \rangle$ and $D = \langle K(B), \{w'_{i} : i < 2n\} \rangle$.  Define an isomorphism $h = (h_{K},h_{V}) : C \to D$ by setting $h_{K} : K(A) \to K(B)$ and $h_{V}$ is the map sending $v'_{i} \mapsto w'_{i}$ for $i < 2n$. 

   Now $h_{V}(V(A))$ and $f_{V}(V(A))$ are isomorphic, via the isomorphism $f_{V} \circ h_{V}^{-1}|_{h_{V}(V(A))}$.  Hence, by Witt's Lemma, there is some $\sigma \in \mathrm{Aut}(D/K(B))$ such that $\sigma$ extends $f_{V} \circ h_{V}^{-1}|_{h_{V}(V(A))}$.  Define $h' : C \to D$ by setting $h' = \sigma \circ h$. Note that 
   $$
   h'|_{V(A)} = \sigma \circ h_{V}|_{V(A)} = f_{V} \circ h^{-1} \circ h_{V}|_{V(A)} = f_{V},
   $$
   and also
   $$
   h'|_{K(A)} = \sigma \circ h_{K}|_{K(A)} = f_{K},
   $$
   since $\sigma$ fixes $K(B)$.  This shows $h'$ extends $f$ and has $c$ in its domain.  This completes the proof.
\end{proof}

The following is an immediate corollary:

\begin{cor}
    Let $T$ be a completion of $T^{0}_{O}$ which asserts $V$ is infinite dimensional and which is either pseudo-finite or asserts $K\models \mathrm{ACF}_{2}$.  Then $T$ eliminates quantifiers down to those on the field (so eliminates quantifiers outright in the case that $K$ is algebraically closed).  
\end{cor}

From elimination of quantifiers, we obtain stable embeddedness of the field sort:

\begin{cor} \label{cor:orthogonalstableembeddedness}
    Let $M = (V,K,q) \models T^{0}_{O}$ with $K$ an algebraically closed or pseudo-finite field of characteristic $2$ and $V$ infinite dimensional. Suppose $C$ is a substructure of $M$ and $a$ and $b$ are tuples from the field sort with $a \equiv_{K(C)} b$ in $\mathrm{Th}(K)$. Then $a \equiv_{C} b$ in $M$. 
\end{cor}

\begin{proof}
    Let $K_{0} = \mathrm{acl}(aK(C))$ and $K_{1} = \mathrm{acl}(bK(C))$, with algebraic closure here computed in $\mathrm{Th}(K(\mathbb{M}))$. By Lemma \ref{lem:substructure}, $\langle K_{i},C \rangle = (K_{i},K_{i} \otimes_{K(C)} C)$ for $i = 0,1$.  Let $\sigma_{K} : K_{0} \to K_{1}$ be a partial elementary isomorphism of $K_{0}$ and $K_{1}$ over $K(C)$ with $\sigma_{K}(a) = b$. Then define an isomorphism $\sigma_{V} : K_{0} \otimes_{K(C)} C \to K_{1} \otimes_{K(C)} C$ to be the unique map that sends $\alpha \otimes c\mapsto \sigma(\alpha) \otimes c$ for all $\alpha \in K_{0}$ and $c \in C$. Then $\tilde{\sigma} = (\sigma_{K},\sigma_{V})$ defines a $C$-isomorphism from $\langle K_{0},C \rangle$ to $\langle K_{1},C \rangle$ which is partial elementary in the field sort. Hence by quantifier elimination, we have $a \equiv_{C} b$. 
\end{proof}

\begin{prop} \label{prop:orthclass}
    Suppose $M \models T^{0}_{O}$ is pseudo-finite.  Then $\mathrm{Th}(M)$ is axiomatized, modulo $T^{0}_{O}$, by one of the following:
    \begin{enumerate}
        \item $K$ is a finite field of cardinality $2^{k}$ and $V$ has infinite dimension.  
        \item $\mathrm{Th}(K)$ is a completion of the theory of pseudo-finite fields and one of the following holds:
        \begin{enumerate}
            \item $V$ is of finite dimension $2n$.  In this case, the theory also specifies the Witt defect of $V$ as either $0$ or $1$. 
            \item $V$ is of infinite dimension. 
        \end{enumerate}
    \end{enumerate}
    Moreover, all of these possibilities are realized. 
\end{prop}

\begin{proof}
    First, we consider the case that $\mathrm{dim}(V)$ is finite. By \cite[Proposition 19.17]{aschbacher2000finite}, it follows that the dimension is even, so we will assume it has dimension $2n$.  It is easy to see that, in this case, $M = (V,K,q)$ is interpretable in $K$, hence is pseudo-finite. Consider the following sentence $\phi_{n}$:
\begin{eqnarray*}
(\exists v_{1}, \ldots, v _{2n} \in V)\Large[\theta_{2n}(v_{1},\ldots, v_{2n}) &\wedge& \bigwedge_{i=1}^{2n} q(v_{i}) = 0 \\
&\wedge& \bigwedge_{i = 1}^{n} \beta(v_{2i-1},v_{2i}) = 1 \\
&\wedge& \bigwedge_{i = 1}^{n} \bigwedge_{k \not\in \{2i-1,2i\}} \beta(v_{2i-1},v_{k}) = \beta(v_{2i},v_{k}) = 0 \Large].
\end{eqnarray*}
Note that $M \models \phi_{n}$ if and only if $V$ has a decomposition into an orthogonal direct sum of $n$ hyperbolic planes, which entails that the Witt defect is $0$.  So we need to argue that $T_{O}^{0}$, together with the complete theory of $K$, the sentence that $\mathrm{dim}(V) = 2n$, and the specification of the truth value of $\phi_{n}$ gives a complete theory.  If we call this theory $T$ and take $M_{1},M_{2} \models T$, by taking ultrapowers and applying the Keisler-Shelah theorem \cite[Theorem 6.1.15]{keisler1990model}, we can assume that the fields are isomorphic.  So we have two $2n$ dimensional orthogonal spaces over the same perfect field with the same Witt defect, hence are isomorphic by \cite[Proposition 21.2]{aschbacher2000finite} (we note that the statement in \cite{aschbacher2000finite} is for finite fields, but the proof goes through unchanged for perfect fields with a unique quadratic Galois extension). 

Now we consider the infinite dimensional case.  The case when the field is finite has already been considered by \cite{cherlin2003finite}, and in any case is strictly easier than the case when the field is infinite, since in this case the field cannot grow in extensions. So we let $T'$ be theory axiomatized, modulo $T^{0}_{O}$, by specifying the complete theory of $K$ (a pseudo-finite field) and that $\mathrm{dim}(V) = \infty$.  This is complete by Proposition \ref{prop:orthogonal bnf}.  

To see that every possibility is realized, suppose $K$ is a finite or pseudo-finite field of characteristic $2$.  Then there is a unique quadratic Galois extension $F/K$.  We have seen that $F$, viewed as a $2$-dimensional $K$-vector space, together with the norm map $N^{F}_{K}$, is an orthogonal space with Witt defect $1$. Additionally, we know there is a hyperbolic plane $(H,q)$ over $K$.   By taking an arbitrary finite number of orthogonal direct sums of $H$ or an arbitrary finite number of orthogonal direct sums of $H$ with one $(F,N^{F}_{K})$ summand, we may obtain all the possibilities described in (1) and (2)(a).  Possibility (2)(b) can be obtained by taking an infinite orthogonal direct sum of copies of $H$. 
\end{proof}

\subsection{Approximation and relative homogeneity}

Witt's Lemma, applied as in the proof of Proposition \ref{prop:orthogonal bnf}, may be used to also give $\aleph_{0}$-categoricity, homogeneity, and smooth approximability relative to the field:

\begin{thm} \label{thm:relcat}
 Let $T$ be either $T_{O}$ or a pseudo-finite completion of $T^{0}_{O}$ with $V$ infinite dimensional. 
\begin{enumerate}
    \item (Relative $\aleph_{0}$-categoricity) Suppose $M, N \models T$ are orthogonal spaces with $K(M) \cong K(N)$. If both $V(M)$ and $V(N)$ are $\aleph_{0}$-dimensional, then $M \cong N$.
    \item (Relative homogeneity) Suppose $M \models T$ with $V(M)$ $\aleph_{0}$-dimensional. If $f : A \to B$ is an isomorphism of two substructures of $M$ which are finitely generated over $K(M)$, then there is $\sigma \in \mathrm{Aut}(M/K(M))$ extending $f$. 
    \item (Relative smooth approximability) Suppose $M \models T$ with $V(M)$ $\aleph_{0}$-dimensional. For each $i < \omega$, choose substructures $M_{i} \subseteq M$ a substructure with $K(M_{i})  = K(M)$ and with $V(M_{i})$ a $2i$-dimensional non-degenerate subspace of $V(M)$ such that $M_{i} \subseteq M_{i+1}$ and $M = \bigcup_{i < \omega} M_{i}$. Then, for each $i$ and for all finite tuples $a,b \in M_{i}$, $a$ and $b$ are in the same $\mathrm{Aut}(M/K(M))$ orbit if and only if they are in the same $\mathrm{Aut}_{\{M_{i}\}}(M/K(M))$ orbit. 
\end{enumerate}
Here $\mathrm{Aut}_{\{M_{i}\}}(M/K(M))$ refers to the elements of $\mathrm{Aut}(M/K(M))$ that fix $M_{i}$ setwise. 
\end{thm}

\begin{proof}
    (1) Without loss of generality, we may assume $K(M)$ and $K(N)$ are equal to some field $K_{*}$. Since $V(M)$ and $V(N)$ are $\aleph_{0}$-dimensional, we know that, in order to construct an isomorphism from $M$ to $N$, it suffices to show that if $f: A \to B$ is an isomorphism of finite dimensional $K_{*}$-subspaces of $V(M)$ and $V(N)$ respectively, and $v \in V(M) \setminus A$, then there are finite dimensional subspaces $A' \supseteq A$ and $B' \supseteq B$ of $V(M)$ and $V(N)$ with $v \in A'$ and an isomorphism $f' :A' \to B'$ extending $f$. As in Case 3 in the proof of Proposition \ref{prop:orthogonal bnf}, we let $A'$ be a finite dimensional hyperbolic subspace of $V(M)$ containing $v$ and $A$ and let $B'$ be a hyperbolic subspace of $V(N)$ containing $B$ of the same dimension as $A'$. Then there is a $K_{*}$-isomorphism $g: A' \to B'$.  Then there is an isomorphism $h: g(A) \to B$ given by $h = f \circ g^{-1}$.  By Witt's Lemma (Fact \ref{Witt}), $h$ extends to a $K_{*}$-automorphism $\sigma$ of $B'$. Then $f' = \sigma \circ g : A' \to B'$ is an isomorphism from $A'$ to $B'$ and, for all $a \in A$, we have 
    $$
    f'(a) = \sigma(g(a)) = (f \circ g \circ g^{-1})(a) = f(a),
    $$
    so $f'|_{A} = f$ as desired. 

    (2) This follows from (1), taking $M = N$, since substructures of $M$ that are finitely generated over $K(M)$ correspond to finite dimensional $K(M)$-subspaces of $V(M)$. 

    (3) Suppose $a$ and $b$ are finite tuples from $M_{i}$ which are in the same $\mathrm{Aut}(M/K(M))$-orbit. Let $A$ and $B$ be the substructures of $M_{i}$ generated over $K(M)$ by $a$ and $b$ respectively.  The $K(M)$-isomorphism $A \to B$ determined by mapping $a \mapsto b$ extends to a $K(M)$-automorphism $\sigma$ of $M_{i}$ by Witt's Lemma.  Then $\sigma$ extends to some $\tilde{\sigma} \in \mathrm{Aut}_{\{M_{i}\}}(M/K(M))$ by (2).  This shows that $a$ and $b$ are in the same $\mathrm{Aut}_{\{M_{i}\}}(M/K(M))$ orbit. 
\end{proof}

\begin{rem} \label{rem:slight}
    The proof of Theorem \ref{thm:relcat}(1), actually establishes the slightly stronger result:  suppose $M, N \models T^{0}_{O}$ are orthogonal spaces with $K(M) \cong K(N)$ perfect fields of characteristic $2$. If both $V(M)$ and $V(N)$ are $\aleph_{0}$-dimensional, then $M \cong N$.  
\end{rem}

\subsection{The Granger example in odd characteristic over pseudo-finite fields}

Although this is not the main focus of the paper, we note that the argument of Proposition \ref{prop:orthogonal bnf} goes through essentially unchanged for the two-sorted theory of an infinite dimensional vector space over a pseudo-finite field of characteristic not $2$ equipped with a symmetric or alternating bilinear form.  As this has not appeared explicitly in the literature, we sketch the argument.  It will be used in Section \ref{section:neostability} to show that this theory is NSOP$_{1}$. For this subsection, let $L$ denote the language $L_{V,K}$ together with a binary function symbol $\beta: V^{2} \to K$ (which may be equivalently described as the sublanguage of $L_{O}$ obtained by forgetting $q$).  Following the notation of \cite{granger1999stability}, for a theory $T_{0}$ of fields of odd characteristic, the $L$-theories $_{S}T^{T_{0}}_{\infty}$ and $_{A}T^{T_{0}}_{\infty}$ are the two-sorted theories of infinite dimensional vector spaces over a field $K \models T_{0}$ equipped with a non-degenerate symmetric or alternating bilinear form, respectively. 

\begin{prop}
    Suppose $M, N$ are both models of $_{S}T^{PSF}_{\infty}$ or $_{A}T^{PSF}_{\infty}$ with $K(M),K(N)$ $\aleph_{0}$-saturated and elementarily equivalent pseudo-finite fields of characteristic not $2$ (as $L_{\mathrm{ring}}$-structures) and with both $V(M)$ and $V(N)$ infinite dimensional.  Then the set of isomorphisms $f:A \to B$ with $A$ and $B$ between finitely generated substructures of $M$ and $N$ such that the induced map $f|_{K(A)}$ is partial elementary (as an $L_{\mathrm{ring}}$-embedding) has the back and forth property. 
\end{prop}

\begin{proof}
    We will follow the proof of Proposition \ref{prop:orthogonal bnf} and merely indicate where the argument needs to be adjusted. Suppose we are given an isomorphism $f: A \to B$ as in the statement, where $A \subseteq M$ and $B \subseteq N$ are finitely generated substructures and we write $f = (f_{K},f_{V})$ where $f_{K} : K(A) \to K(B)$ is a partial elementary isomorphism of fields.  Fix some $c \in M \setminus A$ and we will show that we can extend $f$ to some $f'$ whose domain is $\langle A c \rangle$. We must consider three cases:

    \textbf{Case 1:  }$c \in K(M)$.

    This case goes through unchanged, extending the map to to the extension of scalars to $K(A)(c)$, using the fact that the bilinear form extends canonically to this extension. 

    \textbf{Case 2:  }$c \in \mathrm{Span}_{V(M)}(V(A))$.  

    As before, in this case, we have $c = \sum_{i = 1}^{n} \alpha_{i} v_{i}$ where each $\alpha_{i} \in K(M)$, and we may iteratively apply Case 1, to extend $f$ to some $f'$ whose domain is $\langle A, \alpha_{1}, \ldots, \alpha_{n} \rangle$, and note that $c \in \langle A, \alpha_{1}, \ldots, \alpha_{n} \rangle$. 

    \textbf{Case 3:  }$c \in V(M) - \mathrm{Span}_{V(M)}(V(A))$.  

    Let $v_{0}, \ldots, v_{k-1}$ be a basis for $V(A)$ and let $w_{i} = f_{V}(v_{i})$. So $w_{0}, \ldots, w_{k-1}$ is also a basis of $V(B)$. By Lemma \ref{lem:hyperbolic}, we can extend $v_{0}, \ldots, v_{k-1}$ to a $K(M)$-basis $v_{0}, \ldots, v_{2n-1}$ of a hyperbolic subspace of $V(M)$ which contains $c$ in its span. Likewise, we can extend $w_{0}, \ldots, w_{k-1}$ to a $K(M)$-basis $w_{0}, \ldots, w_{2n-1}$ of a hyperbolic subspace of $V(M)$. 
    
    By Case 1, we may assume that, for all $i < k$, $\beta(c,v_{i}) \in K(A)$ and hence $\beta(c,v) \in K(A)$ for all $v \in V(A)$, and also that $\beta(c,c) \in V(A)$ (which is automatic in the case that $\beta$ is alternating).  By non-degeneracy, we may find some $d \in V(N)\setminus \mathrm{Span}_{V(N)}(V(B))$ such that, for all $w \in V(B)$,
    $$
    \beta(d,w) = f_{K}(\beta(c,f^{-1}_{V}(w)))
    $$
    and $\beta(d,d) = f_{K}(\beta(c,c))$.  Then the map $f'$ extending $f$ by mapping $c \mapsto d$ is the desired isomorphism. 
\end{proof}

We then obtain the parallel corollaries:

\begin{cor} \label{cor:grangerqe}
    If $T_{0}$ is any completion of $\mathrm{PSF}_{\neq 2}$, then $_{S}T^{T_{0}}_{\infty}$ and $_{A}T^{T_{0}}_{\infty}$ eliminate quantifiers down to those on the field.  
\end{cor}

\begin{cor} \label{cor:grangerstableembeddedness}
    Let $M$ be a model of $_{S}T^{PSF}_{\infty}$ or $_{A}T^{PSF}_{\infty}$ with $K$ of odd characteristic and $V$ infinite dimensional. Suppose $C$ is a substructure of $M$ and $a$ and $b$ are tuples from the field sort with $a \equiv_{K(C)} b$ in $\mathrm{Th}(K)$. Then $a \equiv_{C} b$ in $M$.  
\end{cor}

\section{Quadratic Geometries}

\subsection{The basic theory}  The language $L_{Q}$ will extend the language $L_{V,K}$ with the addition of a new sort $Q$, as well as function symbols $\beta_{V} : V^{2} \to K$, $+_{Q} : Q \times V \to Q$, $-_{Q} : Q^{2} \to V$, and $\beta_{Q} : Q \times V \to K$.  

The theory $T^{0}_{Q}$ extends $T^{0}_{V,K}$ and asserts that $K$ has characteristic $2$.  There are axioms asserting that $\beta_{V}$ is a non-degenerate alternating bilinear form on $V$ and that, for each $q \in Q$, $\beta_{Q}(q,\cdot) : V \to K$ is a quadratic form on $V$ with associated bilinear form $\beta_{V}$, i.e. for all $x,y \in V$ and $\alpha \in K$, we have $\beta_{Q}(q,\alpha x) = \alpha^{2}\beta_{Q}(q,x)$ and 
$$
\beta_{V}(x,y) = \beta_{Q}(q,x+y) - \beta_{Q}(q,x) - \beta_{Q}(q,y).  
$$
The theory asserts that $+_{Q}$ defines a regular action of $V$ on $Q$ with the property that, for all $q \in Q$ and $v,w \in V$, 
$$
\beta_{Q}(q+_{Q} v,w) = \beta_{Q}(q,w) + \beta_{V}(v,w)^{2}.  
$$
There is also an axiom stating that, for all $q,q' \in Q$, $q -_{Q} q' = v$ for the unique $v$ with $q' +_{Q} v = q$.  We will call models of $T^{0}_{Q}$ \emph{quadratic geometries}. 

We will also work with a further expansion $L_{Q,\omega}$, which extends $L_{Q}$ with the addition of a unary function $\omega : Q \to K$. The theory $T^{0}_{Q,\omega}$ has an additional axiom asserting that $\omega(q) \in \{0,1\}$ and, for all $q_{1},q_{2} \in Q$,
$$
\omega(q_{1}) = \omega(q_{2}) \iff (\exists x \in K)( q_{1}(q_{1}-_{Q} q_{2}) = x^{2} - x).  
$$
Moreover, $T^{0}_{Q,\omega}$ will have an axiom for each $n$ asserting that if $V$ has dimension $2n$, then $\omega$ outputs the Witt index of $q$ (see Remark \ref{rem: Wittdef} below).  We will refer to models of $T^{0}_{Q,\omega}$ as \emph{quadratic geometries with Witt defect}.  

It may be useful to give a gloss on what is going in $T^{0}_{Q,\omega}$.  The idea is that $Q$ denotes a set of quadratic forms on $V$ which give rise to the bilinear form $\beta_{V}$.  The function $\beta_{Q}$ is the `evaluation' function, so we may identify each $q \in Q$ with the quadratic form $\beta_{Q}(q,\cdot): V \to K$.  Then, given any two forms $q,q' \in Q$, the function $\sqrt{q - q'}$ is a linear functional on $V$.  The axioms of $T^{0}_{Q,\omega}$ say that we can write this function in the form $\lambda_{v} := \beta_{V}(v,\cdot)$ so we get a regular action of $V$ on $Q$ where $v$ acts on $q$ by sending it to $q + \lambda_{v}^{2}$.  This action is given by $+_{Q}$.  Finally, letting $\wp(x) = x^{2} - x$ be the Artin-Schrier map, Cherlin and Hrushovski define a $K$-valued binary function $[q_{1},q_{2}]$ on $Q$ by stipulating that if $q_{2} = q_{1} + \lambda_{v}^{2}$, then $[q_{1},q_{2}] = q_{1}(v) = q_{2}(v)$.  Then the relation $[q_{1},q_{2}] \in \wp(K)$ defines an equivalence relation on $Q$.  The axiom that $\omega(q) \in \{0,1\}$, then, implies that this equivalence relation has at most $2$ classes. 

If $M \models T^{0}_{Q}$ and $q_{*} \in Q$, we may naturally interpret an orthogonal space $(V(M),K(M),q_{*})$, which may be viewed as an $L_{O}$-structure with the function symbol $q$ interpreted to be $\beta_{Q}^{M}(q_{*},-)$.  We will make frequent use of this interpretation to lift arguments for quantifier elimination from the orthogonal space situation to the case of orthogonal geometries.

\begin{rem} \label{rem: Wittdef}
The proof of Proposition \ref{prop:orthclass} establishes that if $M \models T^{0}_{Q}$ is a finite dimensional model with $K(M)$ a perfect field having a unique quadratic Galois extension, the Witt defect is definable (with quantifiers) in $L_{Q}$.  Indeed, suppose that $V(M)$ has dimension $2n$.  Then expresses that $\omega(q) = 0$ if and only if $(V(M),q)$ is hyperbolic, that is 
    \begin{eqnarray*}
(\exists v_{1}, \ldots, v _{2n} \in V)\Large[\theta_{2n}(v_{1},\ldots, v_{2n}) &\wedge& \bigwedge_{i=1}^{2n} q(v_{i}) = 0 \\
&\wedge& \bigwedge_{i = 1}^{n} \beta(v_{2i-1},v_{2i}) = 1 \\
&\wedge& \bigwedge_{i = 1}^{n} \bigwedge_{k \not\in \{2i-1,2i\}} \beta(v_{2i-1},v_{k}) = \beta(v_{2i},v_{k}) = 0 \Large],
\end{eqnarray*}
where we write $q(v)$ as an abbreviation for $\beta_{Q}(q,v)$. Then $\omega(q) = 1$ if and only if $\omega(q) \neq 0$.  This gives an $L_{Q}$-definition of the Witt defect function $\omega$. 
\end{rem}

\subsection{Extension of scalars}  

Suppose $M \models T^{0}_{Q}$ is a quadratic geometry.  Given a field $K'$ extending $K(M)$, we want to describe how to extend scalars from $K(M)$ to $K'$ to obtain an $L_{Q}$-extension of $M$.  We will also consider how to do this in the case that $M$ is a quadratic geometry with Witt defect, though in this case we will have to impose restrictions on which fields $K$ we may use as extension fields. 

Fix some $M \models T^{0}_{Q}$ with $Q(M) \neq \emptyset$ and fix some $q_{0} \in Q(M)$. Then as $V(M)$ acts regularly on $Q(M)$, we know that $Q(M) = \{v +_{Q} q_{0} : v \in V(M)\}$.  Writing $q_{v}$ for $v +_{Q} q_{0}$, we can think of the action of $V(M)$ on $Q(M)$ as an action on subscripts: $w +_{Q} q_{v} = q_{w+v}$.  If $K'$ is a field extension of $K(M)$, we define the \emph{extension of scalars} $M' = (V',K',Q')$ \emph{of} $M$ \emph{to} $K'$ to be the orthogonal geometry with vector space $V' = K' \otimes_{K(M)} V(M)$ and with quadratic form sort $\{q_{v} : v \in V'\}$.  Viewing $V(M)$ as a subset of $V'$ by identifying $v$ with $1 \otimes v$ for all $v \in V(M)$, we can thus view $Q(M) = \{q_{v} : v \in V(M)\}$ as a subset of $Q' = \{q_{v} : v \in V'\}$ in a natural way. The bilinear form $\beta_{V}$ and the quadratic form $q_{0}$ extend canonically to $V'$.  Then we interpret the evaluation of $q_{v}$ on $V'$ as $q_{0} + \lambda_{v}^{2}$.  It is easy to check that this did not depend on the choice of $q_{0}$. 

The following lemma establishes that extension of scalars coincides with the substructure of the monster generated by additional field elements: 

\begin{lem} \label{lem:quadraticextensionofscalars}
Suppose $C \subseteq M \models T^{0}_{Q}$ is a substructure, $q_{0} \in Q(C)$, and $K'$ is a field with $K(C) \subseteq K' \subseteq K(M)$.  Then $\langle K',C \rangle = (V',K',Q')$ where $V' = K' \otimes_{K(C)} V(C)$ and $Q' = \{v +_{Q} q_{0} : v \in V'\}$.  
\end{lem}

\begin{proof}
    As in Lemma \ref{lem:substructure}, we observe that the structure $(K',V',Q')$ clearly contains $K'$ and $C$ and is closed under the functions of $L_{Q}$. 
\end{proof}

To deal with extension of scalars in the case that Witt defect is present, we will have to take Artin-Schreier extensions into account.

\begin{lem} \label{lem:finding one form}
    Suppose $M \models T^{0}_{Q}$.  If $K(M)$ is a perfect field, $A \subseteq M$ is a finitely generated substructure, and $q_{*} : V(A) \to K(A)$ is a quadratic form associated to $\beta_{V}|_{V(A)^{2}}$, then there is $q \in Q(M)$ such that $q|_{V(A)} = q_{*}$, i.e. 
    $$
    M \models \beta_{Q}(q,v) = q_{*}(v)
    $$
    for all $v \in V(A)$.  
\end{lem}

\begin{proof}
    Let $w_{1}, \ldots, w_{k}$ be a basis of $V(A)$ and let $q_{0} \in Q(M)$ be arbitrary.  Define $\lambda : V(A) \to K(M)$ by 
    $$
    \lambda(w) = \sqrt{q_{*}(w) - q_{0}(w)},
    $$
    which is well-defined since $K(M)$ is perfect and extends canonically to a $K(M)$-linear functional on $K(M) \otimes_{K(A)} V(A) = V(\langle K(M),A \rangle)$.  Since $\beta_{V}$ is non-degenerate, there is some $v \in V(M)$ such that $\lambda_{v}(w) = \lambda(w)$ for all $w \in V(A)$, where $\lambda_{v} = \beta_{V}(v,-)$.  Now, unraveling, we see that $(q_{0} + \lambda^{2}_{v})|_{V(A)} = q_{*}$.  Defining $q = q_{0} +_{Q} v \in Q(M)$, then, gives us the desired quadratic form.
\end{proof}

\begin{lem} \label{lem: one AS extn}
    Suppose $M$ is a model of $T^{0}_{Q,\omega}$ with $K(M)$ a perfect field and $Q(M)$ non-empty. Then the field $K$ has at most one Artin-Schrier extension.
\end{lem}

\begin{proof}
    Suppose towards contradiction that, over $K(M)$, there are two irreducible Artin-Schreier polynomials $X^{2} - X - \alpha$ and $X^{2} - X - \alpha'$ which give distinct extensions. Clearly we have $\alpha,\alpha' \not\in \wp (K)$. Moreover, $\alpha + \alpha' \not\in \wp(K(M))$, since if $\alpha' = \alpha + \wp(\gamma)$ for some $\gamma \in K(M)$, then it is easy to check that, for any root $\delta$ of $X^{2} - X - \alpha$, $\delta + \gamma$ is a root of $X^{2} - X - \alpha'$, hence adjoining a root of one polynomial adjoins a root to both, against our assumption that these polynomials give distinct Artin-Schrier extensions. Let $\delta$ and $\delta'$ denote roots of $X^{2} - X - \alpha$ and $X^{2} - X - \alpha'$ respectively.  Let $F = K(M)(\delta)$ and $F' = K(M)(\delta)$.  We know that, viewing $F$ and $F'$ each as vector spaces over $K(M)$, $(F,N^{F}_{K(M)})$ and $(F',N^{F}_{K(M)})$ are both $2$-dimensional orthogonal spaces over $K(M)$. 
    
    Let $q_{0} \in Q(M)$ be arbitrary and find $v,w \in V(M)$ such that $\{v,w\}$ is a hyperbolic basis for a hyperbolic plane, i.e. $q_{0}(v) = q_{0}(w) = 0$ and $\beta_{V}(v,w) = 1$. Let $A$ be the substructure of $M$ generated by $K(M)$, $v,w$, and $q_{0}$. By Lemma \ref{lem:finding one form}, we can find $q,q' \in Q(A)$ that mimic the orthogonal spaces $(F,N^{F}_{K(M)})$ and $(F',N^{F}_{K(M)})$ with respect to the basis $\{v,w\}$.  More precisely, we have $q(v) = q'(v) = 1$ and $q(w) = N^{F}_{K(M)}(\delta) = \alpha$ and $q'(w) = N^{F}_{K(M)}(\delta') = \alpha'$.  To get a contradiction, we will show that $\omega$ cannot take the same value on any two of $\{q_{0}, q,q'\}$. 
    
    Let the vector $u = \sqrt{\alpha} v + w \in V(A)$. Then we have $q = q_{0} + \lambda_{u}^{2}$. To see this, we calculate 
    \begin{eqnarray*}
        (q_{0} + \lambda_{u}^{2})(v) &=& q_{0}(v) + (\beta_{V}(\sqrt{\alpha}v, v) + \beta_{V}(w,v))^{2} = 1 = q(v) \\
        (q_{0} + \lambda_{u}^{2})(w) &=& q_{0}(w) + (\sqrt{\alpha}\beta_{V}(v,w) + \beta_{V}(w,w))^{2} = \alpha = q(w).
    \end{eqnarray*}
    An identical argument shows $q' = q_{0} + \lambda^{2}_{u'}$ where $u' = \sqrt{\alpha'}v + w$. 

    Now we show that $\omega(q_{0}) \neq \omega(q)$. For this, we must show that $q_{0}(u) \not\in \wp(K(M))$. We have 
    $$
    q_{0}(u) = \alpha q_{0}(v) + q_{0}(w) + \beta_{V}(\sqrt{\alpha}v,w) = \sqrt{\alpha}. 
    $$
    Note that if $\sqrt{\alpha} = \gamma^{2} + \gamma$ for some $\gamma \in K(M)$, then $\alpha = \gamma^{4} + \gamma^{2} = \wp(\gamma^{2})$, so we have $q_{0}(u) = \sqrt{\alpha} \not\in \wp(K(M))$ by our assumption that $\alpha \not\in \wp(K(M))$. A symmetric argument shows $q_{0}(u') \not\in \wp(K(M))$. Thus $\omega(q) \neq \omega(q_{0}) \neq \omega(q')$.

    We are left with showing that $\omega(q) \neq \omega(q')$. We have $q = q' + \lambda_{u+u'}^{2}$, we have to show that $q(u+u') \not\in \wp(K(M))$.  But $u + u' = \sqrt{\alpha}v + \sqrt{\alpha'}v = \sqrt{\alpha+\alpha'}v$, so $q(u+u') = \alpha+\alpha' \not\in \wp(K(M))$ as desired.  
\end{proof}
    
%

To conclude our discussion of extension of scalars for models of $T^{0}_{Q,\omega}$, we need to explain how to intrepret $\omega$ on the extension of scalars. Suppose $M \models T^{0}_{Q,\omega}$, $K'$ is a field extending $K(M)$ which contains no Artin-Schreier root of a polynomial over $K(M)$ and has at most one Artin-Schreier extension. Let $M' = (K' \otimes_{K(M)} V(M), K',Q')$ be the extension of scalars of $M$ (viewed as an $L_{Q}$-structure) to $K'$ and we will explain how to interpret $\omega$.  We pick some $q_{0} \in Q(M)$ arbitrary and then it follows that $Q' = \{v +_{Q} q_{0} : v \in V(M')\}$.  Identifying $q_{0}$ with its canonical extension to a quadratic form on $V(M')$, we define, for all $v \in V(M')$, 
$$
 \omega(v +_{Q} q_{0}) = \left\{
\begin{matrix}
    \omega(q_{0}) & \text{ if } q_{0}(v) \in \wp(K') \\
    1- \omega(v+_{Q} q_{0}) & \text{ if } q_{0}(v) \not\in \wp(K')
\end{matrix}
\right.
$$
We now check that this gives a well-defined extension of $M$:

\begin{lem}\label{lem:QG_ext}
    Suppose $M = (V,K,Q) \models T^{0}_{Q,\omega}$ and $K'$ is a field extending $K(M)$ which contains no Artin-Schreier root of a polynomial over $K(M)$ and has at most one Artin-Schreier extension.  
    \begin{enumerate}
        \item The $L_{Q,\omega}$-structure $M'$ is well-defined. 
        \item The natural inclusion $M \to M'$ is an embedding of $L_{Q,\omega}$-structures and $M' \models T^{0}_{Q,\omega}$. 
    \end{enumerate} 
\end{lem}

\begin{proof}
   (1) We must check that our interpretation of $\omega$ on $M'$ did not depend on the choice of $q_{0}$. Suppose we make a different choice of $q'_{0} \in Q(M)$.  Then there is some $v \in V(M)$ such that $q'_{0} = v +_{Q} q_{0}$.  Notice our assumption on $K'$ entails that $\wp(K') \cap K(M) = \wp(K(M))$.  Hence, we have $\omega(q'_{0}) = \omega(q_{0})$ if and only if $q_{0}(v) \in \wp(K(M))$ if and only if $q_{0}(v) \in \wp(K')$.  

   Since $\omega$ only takes on values $\{0,1\}$, we want to show that $\omega(w+_{Q} q'_{0}) = \omega(q'_{0})$ if and only if $q'_{0}(w) \in \wp(K')$, for all $w \in V(M')$. Now we note that $w+_{Q} q'_{0} = (v+w)+_{Q} q_{0}$ and we have 
   \begin{eqnarray*}
       \omega((v+w)+_{Q} q_{0}) = \omega(q_{0}) &\iff& q_{0}(v+w) \in \wp(K') \\
       &\iff& q_{0}(v)+q_{0}(w) + \beta_{V}(v,w) \in \wp(K') \\
       &\iff& q_{0}(v) + q_{0}(w) + 2 \beta_{V}(v,w) + \beta_{V}(v,w)^{2} \in \wp(K') \\
       &\iff& q_{0}(v)+q_{0}(w) + \beta_{V}(v,w)^{2} \in \wp(K'),
   \end{eqnarray*}
   where the third biconditional follows since $\beta_{V}(v,w) + \beta_{V}(v,w)^{2} \in \wp(K')$ and $\wp(K')$ is an additive subgroup. This establishes 
   $$
   \omega(w +_{Q} q'_{0}) = \omega(q_{0}) \iff q_{0}(v)+q_{0}(w) + \beta_{V}(v,w)^{2} \in \wp(K').
   $$
   Moreover, since $q'_{0} = v+_{Q} q_{0}$, we have $q'_{0}(w) = q_{0}(w) + \beta_{V}(v,w)^{2}$. 
   
   Putting it all together, we see that if $q_{0}(v) \in \wp(K')$, then $\omega(w+_{Q}q'_{0}) = \omega(q_{0}) = \omega(q'_{0})$ if and only if $q'_{0}(w) \in \wp(K')$ as desired. 
   
   On the other hand, if $q_{0}(v) \not\in \wp(K')$, then we have $\omega(w+_{Q} q'_{0}) = \omega(q'_{0})$ if and only if $q_{0}(v)+q_{0}(w)+ \beta_{V}(v,w)^{2} \not\in \wp(K')$ if and only if $q_{0}(w)+ \beta_{V}(v,w)^{2} \in \wp(K')$ (since $\wp(K')$ has index $2$ in $(K',+)$). This shows $\omega(w+_{Q}q'_{0}) = \omega(q'_{0})$ if and only if $q'_{0}(w) \in \wp(K')$ and thus $\omega$ is well-defined.

   (2) The natural map $M \to M'$ is clearly an embedding of $L_{Q}$-structures so we just need to check that it respects $\omega$. But the above proof defined $\omega$ so that $\omega(v+_{Q} q_{0}) = \omega(q_{0})$ if and only if $q_{0}(v) \in \wp(K')$ and, as noted above, our assumption on $K'$ entails that $\wp(K') \cap K(M) = \wp(K(M))$.  Hence, for vectors $v \in V(M)$, in $M'$ we have $\omega(v+_{Q} q_{0}) = \omega(q_{0})$ if and only if $q_{0}(v) \in \wp(K)$ and thus $\omega$ on $M'$ agrees with its interpretation on $M$.
\end{proof}

\subsection{Some properties of existentially closed quadratic geometries}

\begin{defn}
    A profinite group $G$ is \emph{projective} if whenever $\alpha: G \to A$ and $\beta : B \to A$ are continuous epimorphisms, where $A$ and $B$ are finite groups, then there is a continuous homomorphism $\gamma :G \to B$ so that $\beta \circ \gamma = \alpha$ as in the following diagram:
$$
\xymatrix{ & G\ar@{-->}[dl]  \ar@{->>}[d] \\
B \ar@{->>}[r] & A }
$$
\end{defn}

If $G$ is a projective profinite group and $\pi : F \to G$ is an epimorphism, there there is a homomorphism $\pi' : G \to F$ such that $\pi \circ \pi' = \mathrm{id}_{G}$ (see \cite[Remark 22.4.2]{fried2008field}).  Examples of projective profinite groups include free and free pro-$p$ profinite groups \cite[Proposition 22.10.4]{fried2008field}. 

\begin{fact} \cite[Proposition 22.6.1]{fried2008field}
Given any profinite group $G$, there is a unique smallest projective profinite group $\tilde{G}$ with an epimorphism $\pi: \tilde{G} \to G$.  The epimorphism $\pi$ (or, for simplicity, the group $\tilde{G}$) is called the \emph{universal Frattini cover of} $G$. `Smallest' here means that if $\rho : H \to G$ is an epimorphism from a projective profinite group $H$, then there is some epimorphism $\pi' : H \to \tilde{G}$ with $\pi' \circ \pi = \rho$. 
\end{fact}

\begin{lem}
    If $M = (V,K,Q)$ is an existentially closed model of $T^{0}_{Q,\omega}$, then $K$ is a PAC field with absolute Galois group $G(K) = \mathbb{Z}_{2}$, the free pro-2 group on a single generator.
\end{lem}

\begin{proof}
    By Lemma \ref{lem: one AS extn}, $K$ has at most one Artin-Schreier extension and as $M$ is existentially closed, we know $K$ must have a unique Artin-Schrier extension $K(\alpha)$, so
    $$
    \mathcal{G}al(K(\alpha)/K) \cong \mathbb{Z}/2\mathbb{Z}.
    $$
    Let $K'$ be a regular extension of $K$ that is PAC (which exists by \cite[Proposition 13.4.6]{fried2008field}). The universal Frattini cover of $\mathbb{Z}/2\mathbb{Z}$ is free pro-2 of the same rank (by \cite[Proposition 22.7.6]{fried2008field} and \cite[Proposition 22.10.5]{fried2008field}), i.e. is $\mathbb{Z}_{2}$, so we can let $\pi : \mathbb{Z}_{2} \to \mathbb{Z}/2\mathbb{Z}$ be the universal Frattini cover of $\mathbb{Z}/2\mathbb{Z}$. Since $K'$ is PAC, the absolute Galois group $G(K')$ is projective and, since $K'$ is a regular extension of $K$, the restriction map $\mathrm{res}: G(K') \to \mathcal{G}al(K(\alpha)/K)$ is surjective.  Since the universal Frattini cover of $\mathbb{Z}/2\mathbb{Z}$ is the smallest projective profinite group with an epimorphsim to $\mathbb{Z}/2\mathbb{Z}$, there is some $\rho: G(K') \to \mathbb{Z}_{2}$ with $\mathrm{res} = \pi \circ \rho$. Since $\mathbb{Z}_{2}$ is projective, there is a map $\rho': \mathbb{Z}_{2} \to G(K')$ such that $\rho \circ \rho' = \mathrm{id}_{\mathbb{Z}_{2}}$.  Note that this forces $\rho'$ to be injective.  Let $H = \rho'(\mathbb{Z}_{2})$ and let $K'' = (K'^{s})^{H}$ be the fixed field of $H$.  Note that, as an algebraic extension of a PAC field, $K''$ is PAC (by \cite[Corollary 11.2.5]{fried2008field}). We have $G(K'') \cong \mathbb{Z}_{2}$, $K''$ has a unique Artin-Schrier extension (as $\mathbb{Z}_{2}$ has a unique $\mathbb{Z}/2\mathbb{Z}$ quotient), and the restriction $G(K'') \to \mathcal{G}al(K(\alpha)/K)$ is surjective. 

    Let $M''$ denote the extension of scalars of $M$ to $K''$.  By Lemma \ref{lem:QG_ext}, $M'' \models T^{0}_{Q,\omega}$ and the natural inclusion $M \to M''$ is an embedding of $L_{Q,\omega}$-structures. As every absolutely irreducible $K$-variety is, in particular, an absolutely irreducible $K''$ variety, we have that every absolutely irreducible $K$-variety has a $K''$-rational point and therefore a $K$-rational point, by existential closedness. This shows that $K$ is PAC.  

    Since we know that $K$ is PAC and therefore $G(K)$ is projective, we may repeat the above arguments, replacing $G(K'')$ with $G(K)$, to find an injective map $\rho': \mathbb{Z}_{2} \to G(K)$ whose image $H'$ surjects onto $\mathcal{G}al(K(\alpha)/K)$ via the restriction map. If $H'$ is a proper subgroup, then $K_{*} = (K^{s})^{H}$ is a proper algebraic extension with a unique Artin-Schrier extension, so the extension of scalars $M_{*}$ of $M$ to $K_{*}$ is a model of $T^{0}_{Q,\omega}$ and the natural inclusion $M \to M_{*}$ is an embedding of $L_{Q,\omega}$-structures. This contradicts the existential closedness of $M$.  Therefore, we conclude that $H' = G(K)$ so $G(K) \cong \mathbb{Z}_{2}$. 
\end{proof}

\subsection{The model companion}

In this section, we will define the model companions $T_{Q}$ and $T_{Q,\omega}$ of $T^{0}_{Q}$ and $T^{0}_{Q,\omega}$, respectively.  The theory $T_{Q}$ is axiomatized by $T^{0}_{Q}$ together with axioms asserting that $K \models \mathrm{ACF}_{2}$ and $V$ is infinite dimensional.  The axioms for $T_{Q,\omega}$ are a bit more involved and will be described later. 

\begin{lem} \label{lem:Wittswitch}
    Suppose $M \models T^{0}_{Q,\omega}$ is a model such that $K(M)$ has a unique quadratic Galois extension and $V(M)$ is infinite dimensional.  If $A \subseteq M$ is a finitely generated substructure and $q \in Q(M)$ is a quadratic form, then there is a quadratic form $q' \in Q(M)$ such that $q|_{V(A)} = q'|_{V(A)}$ and $\omega(q') = 1 - \omega(q)$.  
\end{lem}

\begin{proof}
    Because $K(M)$ has a quadratic Galois extension, there is some $\alpha \in K(M) - \wp(K(M))$.  Let $v \in V(M)$ be an element of $V(A)^{\perp}$ such that $q(v) = \alpha$.  To see that this exists, note that, since $V(M)$ is infinite dimensional, there are $w,w' \in V(A)^{\perp}$ which span a hyperbolic plane.  Then 
    $$
    q(\alpha v + v') = \beta_{V}(\alpha v,v') + q(\alpha v) + q(v') = \alpha.
    $$
    Define $q' = q+\lambda_{v}^{2}$.  Then, since $q,v \in M$ and $q' = v +_{Q} q$, we have $q' \in Q(M)$ and, by definition of $\omega$, we have $\omega(q') \neq \omega(q)$, hence $\omega(q') = 1 - \omega(q)$. 
\end{proof}

\begin{prop} \label{prop: quadratic bnf}
    Suppose $M,N \models T^{0}_{Q}$ or $M,N \models T^{0}_{Q,\omega}$, $K(M)$ and $K(N)$ are models of ACF$_{2}$ or elementarily equivalent pseudo-finite fields of characteristic $2$ which are $\aleph_{0}$-saturated, and $V(M)$ and $V(N)$ are both infinite dimensional.  Then the set of isomorphisms between finitely generated structures $A \subseteq M$ and $B \subseteq N$ has the back-and-forth property. 
\end{prop}

\begin{proof}
 Suppose $A \subseteq M$ and $B \subseteq N$ are finitely generated substructures and $f: A \to B$ is an isomorphism.

 First, we consider the case that $Q(A) = Q(B) = \emptyset$. Let $q_{0} : V(A) \to K(A)$ be an arbitrary quadratic form associated to $\beta_{V}|_{V(A)^{2}}$.  By Lemma \ref{lem:finding one form}, there is some $q \in Q(M)$ such that $q$ restricts to induce $q_{0}$ on $V(A)$.  Next, we define a form $q_{0}' :V(B) \to K(B)$ by pushing forward $q_{0}$ along the isomorphism $f$\textemdash that is, we set 
 $$
 q'_{0}(v) = f_{K}(q_{0}(f^{-1}_{V}(v)))
 $$
 for all $v \in V(B)$.  Again, applying Lemma \ref{lem:finding one form}, we obtain some $q' \in Q(N)$ such that $q'$ induces $q'_{0}$ when restricted to $V(B)$.  In the case that $M,N \models T^{0}_{Q,\omega}$, we may, by applying Lemma \ref{lem:Wittswitch}, assume that $\omega(q) = \omega(q')$.  Let $A' = \langle A,q \rangle$ and $B' = \langle B,q' \rangle$.  Note that the field and vector space sorts of $A$ and $A'$ are the same, and likewise for $B$ and $B'$.  Moreover, $f$ extends uniquely to an isomorphism $f' : A' \to B'$ with $q \mapsto q'$, as we may define 
 $$
 f'(q + \lambda^{2}_{v}) = q' + \lambda^{2}_{f(v)}
 $$
 for all $v \in V(A)$.  It is immediate by the definitions that this is an isomorphism.  Therefore, we have reduced to the case of extending $f:A \to B$ when $Q(A)$ and $Q(B)$ are non-empty.  Fix some $q_{*} \in Q(A)$ and let $q_{*}' \in Q(B)$ be defined by $q_{*}' = f(q_{*})$.  

 Suppose $c \in M  - A$ and we will extend $f$ to some $f'$ whose domain is a substructure of $M$ containing $c$.  First, suppose $c$ is in either the field or vector space sorts.  We then let $A_{-}$ be the substructure of the interpreted orthogonal space $(V(M), K(M),q_{*}) \models T_{O}$ generated by $K(A), V(A)$ and $c$.  By Proposition \ref{prop:orthogonal bnf}, there is an isomorphism $f_{-} : A_{-} \to B_{-}$ extending $f$ on these sorts, where $B_{-}$ is a substructure of the interpreted orthogonal space $(V(N), K(N),  q_{*}')$ containing $K(B)$ and $V(B)$.  Letting $A' = \langle A_{-},A\rangle$ and $B' = \langle B_{-},B \rangle$, we have an extension $f' : A' \to B'$ of $f$, defined by taking $f'$ to be $f_{-}$ on the field and vector space sorts, and then setting 
 $$
 f'(q_{*}+ \lambda^{2}_{v}) = q_{*}' + \lambda^{2}_{f_{-}(v)}
 $$
 for all $v \in V(A') = V(A_{-})$.  This gives the desired isomorphism.

 For the remaining case, we suppose $c \in Q(M)$.  Then $c = q_{*} + \lambda^{2}_{v}$ for some $v \in V(M) - V(A)$.  By the previous case, we may extend $f$ to an isomorphism $f:A' \to B'$ for $A'$ a finitely generated substructure of $M$ containing $A$ and $v$.  But then $c$ is in the domain of $f$.  This completes the proof. 
\end{proof}

We, then, obtain an analogue of Corollary \ref{cor:orthogonalstableembeddedness} for quadratic spaces.

\begin{cor} \label{cor:quadraticstableembeddedness}
Let $M = (V,K,Q) \models T^{0}_{Q}$ or $M = (V,K,Q) \models T^{0}_{Q,w}$  with $K$ an algebraically closed or pseudo-finite field of characteristic $2$ and $V$ infinite dimensional. Suppose $C$ is a substructure of $M$ and $a$ and $b$ are tuples from the field sort with $a \equiv_{K(C)} b$ in $\mathrm{Th}(K)$. Then $a \equiv_{C} b$ in $M$. 

\end{cor}

\begin{proof}
    As in the proof of Corollary \ref{cor:orthogonalstableembeddedness}, we let $K_{0} = \mathrm{acl}(aK(C))$ and $K_{1} = \mathrm{acl}(bK(C))$, with algebraic closure here computed in $\mathrm{Th}(K(\mathbb{M}))$ and let $\sigma_{K} : K_{0} \to K_{1}$ be a partial elementary isomorphism of $K_{0}$ and $K_{1}$ over $K(C)$ with $\sigma_{K}(a) = b$.
    
    If $Q(C) = \emptyset$, then the conclusion already follows from the argument of Corollary \ref{cor:orthogonalstableembeddedness}, so we may assume we have some $q_{0} \in Q(C)$.  By Lemma \ref{lem:quadraticextensionofscalars}, we have  $\langle K_{i},C \rangle = (V_{i}, K_{i},Q_{i})$, where $V_{i} = K_{i} \otimes_{K(C)} V(C)$ and $Q_{i} = \{v +_{Q} q_{0} : v \in V_{i}\}$, for $i = 0,1$. Then we define an isomorphism $\sigma_{V} : K_{0} \otimes_{K(C)} C \to K_{1} \otimes_{K(C)} C$ to be the unique map that sends $\alpha \otimes c\mapsto \sigma(\alpha) \otimes c$ for all $\alpha \in K_{0}$ and $c \in C$. Likewise, we define a map $\sigma_{Q} :Q_{0} \to Q_{1}$ by mapping $v +_{Q} q_{0} \mapsto \sigma_{V}(v) +_{Q} q_{0}$. Then $\tilde{\sigma} = (\sigma_{V},\sigma_{K},\sigma_{Q})$ defines a $C$-isomorphism from $\langle K_{0},C \rangle$ to $\langle K_{1},C \rangle$ which is partial elementary in the field sort. Hence by quantifier elimination, we have $a \equiv_{C} b$. 
\end{proof}

\subsection{Pseudo-finite models}

\begin{lem} \label{lem: finite iso}
    Suppose $M,N \models T^{0}_{Q,\omega}$ with $K(M) = K(N)$ a perfect field and $\mathrm{dim}(V(M)) = \mathrm{dim}(V(N)) = 2n$. Then $M \cong N$. 
\end{lem}

\begin{proof}
    Since $K(M) = K(N)$ is a perfect field, we may apply Lemma \ref{lem:finding one form} to find $q \in Q(M)$ and $q' \in Q(N)$ such that $(V(M),q)$ and $(V(N),q')$ are both hyperbolic (and hence $\omega(q) = \omega(q') = 0$). Let $f_{V} : V(M) \to V(N)$ be an isometry between $(V(M),q)$ and $(V(N),q')$. Define $f_{Q} : Q(M) \to Q(N)$ by 
    $$
    f_{Q}(v +_{Q} q) = f_{V}(v) +_{Q} q',
    $$
    or, equivalently considering the elements of the $Q$ sorts as functions, $f_{Q}$ maps $q+\lambda_{v}^{2} \mapsto q' + \lambda^{2}_{f_{V}(v)}$. Setting $f_{K} = \mathrm{id}_{K(M)}$, then defines an isomorphism $f : ( V(M), K(M),Q(M)) \to (V(N), K(N), Q(N))$. 
\end{proof}

\begin{thm}
    Suppose $M \models T^{0}_{Q}$ or $M \models T^{0}_{Q,\omega}$ is pseudo-finite with $Q(M) \neq \emptyset$.  Then $\mathrm{Th}(M)$ is axiomatized, modulo $T^{0}_{Q}$ or $T^{0}_{Q,\omega}$, respectively, by one of the following:
    \begin{enumerate}
        \item The theory asserting that $K$ is a finite field of cardinality $2^{k}$, $V$ has infinite dimension, and $Q(M) \neq \emptyset$.
        \item The theory asserting that $\mathrm{Th}(K)$ is a completion of the theory of pseudo-finite fields, $Q(M) \neq \emptyset$, and one of the following holds:
        \begin{enumerate}
            \item $V$ is of finite dimension $2n$. 
            \item $V$ is of infinite dimension. 
        \end{enumerate}
    \end{enumerate}
    Moreover, all of these possibilities are realized. 
\end{thm}

\begin{proof}
    If $\mathrm{Th}(M)$ falls into case (1), then $\mathrm{Th}(M)$ is pseudo-finite and $\aleph_{0}$-categorical and hence complete by CITE.  Hence we may assume $K(M)$ is an infinite field. Suppose now that $\mathrm{Th}_{L_{\mathrm{rings}}}(K(M)) = T_{K}$, some fixed completion of PSF$_{2}$.
    
    We handle the finite-dimensional case. Suppose we are given $M_{0}, M_{1}$ models of $T^{0}_{Q}$ or of $T^{0}_{Q,\omega}$, which have the property that $K(M_{i}) \models T_{K}$ and $V(M_{i})$ has dimension $2n$, for $i = 0,1$. Recall by Remark \ref{rem: Wittdef}, if the $M_{i}$ are merely assumed to be models of $T^{0}_{Q}$, they have a definition expansion to models of $T^{0}_{Q,\omega}$ so there is no difference in this case. As elementarily equivalent structures have isomorphic ultrapowers, we may, after replacing $M_{0}$ and $M_{1}$ by ultrapowers, assume that $K(M_{0}) = K(M_{1})$. Then $M_{0} \cong M_{1}$ by Lemma \ref{lem: finite iso}.  Then $M_{0} \equiv M_{1}$.  This shows that the theory axiomatized by $T^{0}_{Q}$ (or $T^{0}_{Q,\omega}$) together with the assertions that $K \models T_{K}$ and $V$ has dimension $2n$ is complete. Since $V$ has finite dimension over $K$, it is also easily seen to be interpretable in the pseudo-finite field hence pseudo-finite. 

    The remaining case is where the theory asserts that $K \models T_{K}$ and $V$ has infinite dimension. As an ultraproduct of structures of type (2)(a), with unbounded dimension, it is clearly pseudo-finite. And by Proposition \ref{prop: quadratic bnf}, it is complete. It is easy to see that all the possibilities are realized, so we are done. 
\end{proof}

\subsection{Approximation and relative homogeneity}

\begin{thm} \label{thm:quadrelcat}
 Let $T$ be $T_{Q}$, $T_{Q,\omega}$, or a pseudo-finite completion of $T^{0}_{Q,\omega}$ with $V$ infinite dimensional. 
\begin{enumerate}
    \item (Relative $\aleph_{0}$-categoricity) Suppose $M,N \models T$ with $K(M) \cong K(N)$ perfect fields of characteristic $2$. If $V(M)$ and $V(N)$ are $\aleph_{0}$-dimensional and $Q(M)$ and $Q(N)$ are both non-empty, then $M \cong N$. 
    \item (Relative homogeneity) Suppose $M \models T$ with $V(M)$ $\aleph_{0}$-dimensional. If $f : A \to B$ is an isomorphism of two substructures of $M$ which are finitely generated over $K(M)$, then there is $\sigma \in \mathrm{Aut}(M/K(M))$ extending $f$. 
    \item (Relative smooth approximability) Suppose $M \models T$, and $V(M)$ $\aleph_{0}$-dimensional. For each $i < \omega$, choose substructures $M_{i} \subseteq M$ a substructure with $K(M_{i})  = K(M)$ and with $V(M_{i})$ a $2i$-dimensional non-degenerate subspace of $V(M)$ such that $Q(M_{i}) \neq \emptyset$, $M_{i} \subseteq M_{i+1}$, and $M = \bigcup_{i < \omega} M_{i}$. Then, for each $i$ and for all finite tuples $a,b \in M_{i}$, $a$ and $b$ are in the same $\mathrm{Aut}(M/K(M))$ orbit if and only if they are in the same $\mathrm{Aut}_{\{M_{i}\}}(M/K(M))$ orbit. 
\end{enumerate}
Here $\mathrm{Aut}_{\{M_{i}\}}(M/K(M))$ refers to the elements of $\mathrm{Aut}(M/K(M))$ that fix $M_{i}$ setwise. 
\end{thm}

\begin{proof}
    (1) We may assume $K(M) = K(N) = K_{*}$. Choose $q_{0} \in Q(M)$ and $q'_{0} \in Q(N)$.  Then, by Theorem \ref{thm:relcat}(1) (and by Remark \ref{rem:slight}, in the case that $T = T_{Q,\omega}$), the interpreted orthogonal spaces $(V(M),q_{0})$ and $(V(N),q'_{0})$ are isomorphic over $K_{*}$. Let $f_{V} : V(M) \to V(N)$ denote the $K_{*}$-isomorphism. Then we define a bijection $f_{Q} : Q(M) \to Q(N)$ by 
    $$
    f_{Q}(v +_{Q} q_{0}) = f_{V}(v) +_{Q} q'_{0}.
    $$
    It is easy to check that this is an isomorphism.  

    (2) Suppose we are given $f: A \to B$, a $K(M)$-isomorphism of substructures of $M$, where $A$ and $B$ are finitely generated over $K(M)$. We first consider the case that $Q(A) = Q(B) = \emptyset$. Then, as in the proof of Proposition \ref{prop: quadratic bnf}, we let $q_{0}\in Q(M)$ be an arbitrary quadratic form and we define a form $q_{0}' :V(B) \to K(B)$ by pushing forward $q_{0}$ along the isomorphism $f$\textemdash that is, we set 
 $$
 q'_{0}(v) = q_{0}(f^{-1}(v))
 $$
 for all $v \in V(B)$.  Applying Lemma \ref{lem:finding one form}, we find some $q' \in Q(N)$ such that $q'$ induces $q'_{0}$ when restricted to $V(B)$.  Moreover, in the case that $M,N \models T^{0}_{Q,\omega}$, we may, by applying Lemma \ref{lem:Wittswitch}, assume that $\omega(q) = \omega(q')$.  Setting $A' = \langle A,q \rangle$ and $B' = \langle B,q' \rangle$, we see $f$ extends uniquely to an isomorphism $f' : A' \to B'$ with $q \mapsto q'$, as we may define 
 $$
 f'(v+_{Q} q) = v +_{Q} q' 
 $$
 for all $v \in V(A)$.  

 Next, we note that by (the proof of) Theorem \ref{thm:relcat}(1), the isomorphism of orthogonal spaces $(V(A'),q) \to (V(B'),q')$ extends to a $K(M)$-isomorphism of orthogonal spaces $\sigma :(V(M),q) \to (V(M),q')$.  We define $\tau \in \mathrm{Aut}(M/K(M))$ to be the identity on the field sort and $\sigma$ on the vector space sort.  On the quadratic form sort, $\tau$ is defined so that, for all $v \in V(M)$, 
 $$
 \tau(v +_{Q} q) = \sigma(v) +_{Q} q'. 
 $$
 This gives the desired automorphism.

    (3) Suppose $a$ and $b$ are finite tuples from $M_{i}$ which are in the same $\mathrm{Aut}(M/K(M))$-orbit. Let $A$ and $B$ be the substructures of $M_{i}$ generated over $K(M)$ by $a$ and $b$ respectively. Choose $q \in Q(A)$ such that $(V(A),q)$ is hyperbolic and, in the pseudo-finite case where the Witt defect is present, choose $q$ so that $\omega(q) = 0$).  We define $q' = f(q) \in Q(B)$. Then we know that the orthogonal spaces $(V(M_{i}),q)$ and $(V(M_{i}),q')$ are both hyperbolic (either because $K(M)$ is algebraically closed, or because $f$ respects the Witt defect), and so there is a $K(M)$-isomorphism of orthogonal spaces $g : (V(M_{i}),q) \to (V(M_{i}), q')$.  By Witt's Lemma, we may assume $g$ extends $f$ on $V(M_{i})$. By (the proof of) Theorem \ref{thm:relcat}(1), there is some $K(M)$-isomorphism of orthogonal spaces $\sigma : (V(M),q) \to (V(M),q')$ extending $g$. Now we define $\tau \in \mathrm{Aut}(M/K(M))$ to be the identity on $K(M)$, to be $\sigma$ on the vector space sort, and defined so that 
    $$
    \tau(v +_{Q} q) = \sigma(v) +_{Q} q'
    $$
    for all $v \in V(M)$.  We know that every element of $Q(M_{i})$ has the form $v +_{Q} q$ for some $v \in V(M_{i})$ and $\sigma$ fixes $V(M_{i})$ setwise.  Hence, for all $v \in V(M_{i})$, we have $\tau(v+_{Q} q) = \sigma(v) +_{Q} q' \in Q(M_{i})$. Therefore, $\tau \in \mathrm{Aut}_{\{M_{i}\}}(M)$. Similarly, we know that every element of $Q(A)$ has the form $v +_{Q} q$ for some $v \in V(A)$ and, since $f$ is an isomorphism mapping $q \mapsto q'$, $f(v+_{Q} q) = f(v) +_{Q} q'$.  Since $\sigma$ was chosen to extend $f$ on the vector space sort, we get that $\tau$ extends $f$.  This shows that $a$ and $b$ are in the same $\mathrm{Aut}_{\{M_{i}\}}(M/K(M))$ orbit.
\end{proof}

\begin{rem}
    Theorem \ref{thm:quadrelcat}(1) and (2) also go through unchanged but the statement of Theorem \ref{thm:quadrelcat}(3) becomes false in the case of $M$ being a pseudo-finite model of $T^{0}_{Q}$ with $V(M)$ $\aleph_{0}$-dimensional, since the Witt defect is not present in this case.  The issue is that if $M_{i} \subseteq M$ is a finite dimensional substructure and $q,q' \in Q(M_{i})$ are quadratic forms such that the Witt defect of $(V(M_{i}),q)$ and $(V(M_{i}),q')$ differ, then there can be no automorphism of $M_{i}$ that maps $q \mapsto q'$.  However, the infinite dimensional orthogonal spaces $(V(M),q)$ and $(V(M),q')$ are isomorphic and this is witnessed by an automorphism of $M$ mapping $q \mapsto q'$.  In this situation, we see that $q$ and $q'$ are in the same $\mathrm{Aut}(M/K(M))$ orbit, but not the same $\mathrm{Aut}_{\{M_{i}\}}(M/K(M))$ orbit.  This phenomenon plays an important role in the theory of smoothly approximable structures, showing that smooth approximability is not preserved under reduct. 
\end{rem}

\section{Neostability} \label{section:neostability}

To prove NFOP$_{1}$ and NSOP$_{1}$, it is enough to prove this for the quadratic geometries, since the orthogonal spaces are interpretable in them. For NFOP$_{2}$, we will only consider the case when the associated field is algebraically closed as pseudo-finite fields \cite{beyarslan2010random} and, more generally, all non-separably closed PAC fields have IP$_{n}$ for all $n$ \cite{hempel2016n}. 

\subsection{NFOP$_{2}$}

\begin{defn}
  We say a formula $\varphi(x_{1}, x_{2},x_{3})$ has \emph{FOP}$_{2}$ if there are tuples $(a_{\eta})_{\eta \in \omega^{\omega}}$ and $(b_{i},c_{i})_{i < \omega}$ such that, for all $\eta \in \omega^{\omega}$, $i,j < \omega$, we have 
  $$
  \models \varphi(a_{\eta}, b_{i},c_{j}) \iff j \leq \eta(i). 
  $$
  We say that a theory $T$ is NFOP$_{2}$ if no formula has FOP$_{2}$ modulo $T$. 
\end{defn}

\begin{fact} \label{fact:compositionlemma} \cite[Theorem 2.16]{abdaldaim2023higher}
    Suppose $T$ is an $L$-theory, $L_{0} \subseteq L$, and $T \upharpoonright L_{0}$ is stable. let $\theta(w_{1}, \ldots, w_{n})$ be an $L_{0}$-formula.  Suppose $x,y,$ and $z$ are disjoint tuples of variables and, for each $i = 1, \ldots, n$, $u_{i}$ is an ordered pair from $\{x,y,z\}$ and $f_{i}(u_{i})$ is an $L$-definable function. Then the $L$-formula 
    $$
    \varphi(x,y,z) = \theta(f_{1}(u_{1}), \ldots, f_{n}(u_{n}))
    $$
    is NFOP$_{2}$ modulo $T$. 
\end{fact}

For technical reasons, we will have to prove NFOP$_{2}$ of $T_{Q}$ in a language $L_{Q,q_{*}}$ which adds a constant $q_{*}$ to the $Q$ sort. We let $L_{-}$ be the language $L_{Q}$ without the symbols $\beta_{Q}$ and $\beta_{V}$. The theory $T_{-}$ will assert that $K \models \mathrm{ACF}_{2}$, $V$ is an infinite dimensional vector space over $K$, $+_{Q} : V \times Q \to Q$ is a regular action.  Moreover, there is an axiom asserting that for all $q \in Q$ and $v \in V$, $(v+_{Q} q) -_{Q} q = v$. It is easy to see that $T_{-}$ is stable and complete. 

Let $\mathcal{B}$ denote the following set of $L_{Q,q_{*}}$-terms: $\beta_{Q}(q_{*},x)$, $\beta_{Q}(q_{*},x-_{Q} q_{*})$, $\beta_{Q}(q_{*},x-_{Q}y)$, $\beta_{V}(x,y)$, $\beta_{V}(x,y-_{Q} q_{*})$, and $\beta_{V}(x-_{Q} q_{*}, y-_{Q}q_{*})$. Note that every term in $\mathcal{B}$ is at most binary function valued in $K$. We say that an $L_{Q,q_{*}}$-term $t$ is \emph{acceptable} if $t$ is equal modulo $T_{Q,q_{*}}$ to an $L_{-,q_{*}}$-term $s$ in which some of its field sort variables have been replaced by terms from $\mathcal{B}$. For example, the term $t(x,y,z) = \beta_{V}(x+y,z)$ is acceptable because there is an $L_{-,q_{*}}$ term $s(u,v) = u+v$ such that
$$
T_{Q,q_{*}} \vdash t(x,y,z) = s(\beta_{V}(x,z),\beta_{V}(y,z)).
$$
We will eventually show that every term of $L_{Q,q_{*}}$ is acceptable, which will enable us to apply the Composition Lemma, Fact \ref{fact:compositionlemma}, to conclude that $T_{Q,q_{*}}$, and hence $T_{Q}$, is NFOP$_{2}$. 

\begin{lem} \label{lem:acceptablelem}
    Suppose $b(x,y) \in \mathcal{B}$ and $t,t'$ are $L_{Q,q_{*}}$-terms, valued in the sorts of $x$ and $y$, respectively. Then the term $b(t,t')$ is acceptable (or $b(x) \in \mathcal{B}$ and $b(t)$ is acceptable, in the case that $b$ is unary). 
\end{lem}

\begin{proof}
    We will prove this by induction on the term $b$. We will assume that $t_{0}, t_{1}$, and $t'$ are terms such that the conclusion has been established and then obtain the conclusion for the various choices for $b \in B$ and ways of constructing $t$ from $t_{0}$ and $t_{1}$. 

    In the first case, we will assume $t_{0}, t_{1}$ are $Q$-valued and we will prove the statement for $t = t_{0} -_{Q} t_{1}$ and $t'$. If $b = \beta_{Q}(q_{*},x)$, then we have, modulo $T_{Q,q_{*}}$
    \begin{eqnarray*}
        \beta_{Q}(q_{*}, t_{0}-_{Q}t_{1}) &=& \beta_{Q}(q_{*}, (t_{0}-_{Q}q_{*}) + (q_{*}-_{Q} t_{1})) \\
        &=& \beta_{Q}(q_{*}, t_{0}-_{Q} q_{*}) + \beta_{Q}(q_{*}, t_{1}-_{Q} q_{*}) + \beta_{V}(t_{0}-_{Q}q_{*}, t_{1}-_{Q}q_{*}),
    \end{eqnarray*}
    which shows, by induction, that $b$ is equal to a sum of three acceptable terms and is therefore acceptable.

    Arguing similarly, we have 
    \begin{eqnarray*}
    \beta_{V}(t_{0}-_{Q} t_{1},t') &=& \beta_{V}((t_{0}-_{Q}q_{*}) + (q_{*}-_{Q} t_{1}), t') \\
    &=& \beta_{V}(t_{0}-_{Q}q_{*}, t') + \beta_{V}(t_{1} -_{Q} q_{*},t')
    \end{eqnarray*}
    which is a sum of acceptable terms, hence acceptable. The term $\beta_{V}(t_{0}-_{Q}t_{1}, t'-_{Q} q_{*})$ is acceptable by an identical argument.  This covers all cases where the term $t = t_{0} -_{Q} t_{1}$. 

    The other cases where $t$ is valued in the $V$ sort are $t = t_{0} \cdot t_{1}$ for $t_{0}$ $K$-valued and $t_{1}$ $V$-valued, and $t = t_{0} + t_{1}$, for $t_{0}, t_{1}$ $V$-valued. It is easy to show, in these cases, that $\beta_{Q}(q_{*},t)$, $\beta_{V}(t,t')$ and $\beta_{V}(t,t'-_{Q} q_{*})$ are acceptable, using that $\beta_{Q}(q_{*},x)$ and $\beta_{V}(x,y)$ are quadratic and bilinear forms, respectively.

    Now we consider the second case, where $t$ is valued in the $Q$ sort. Then $t = t_{0} +_{Q} t_{1}$, where $t_{0}$ is $V$-valued and $t_{1}$ is $Q$-valued. We have $4$ terms in $\mathcal{B}$ to handle: $\beta_{Q}(q_{*}, t-_{Q} q_{*})$, $\beta_{Q}(q_{*}, t-_{Q}t')$, $\beta_{V}(t-_{Q}q_{*}, t')$, and $\beta_{V}(t-_{Q} q_{*}, t'-_{Q}q_{*})$.  We will first show $\beta_{Q}(q_{*}, t-_{Q} q_{*})$ is acceptable.  We calculate 
    \begin{eqnarray*}
        \beta_{Q}(q_{*},((t_{0}+_{Q} t_{1})-_{Q} q_{*}) &=& \beta_{Q}(q_{*}, (t_{1}-_{Q} q_{*}) + t_{0}) \\
        &=& \beta_{Q} (q_{*}, (t_{1}-_{Q}q_{*})) + \beta_{Q}(q_{*},t_{0}) + \beta_{V}((t_{1}-_{Q} q_{*}), t_{0}),
    \end{eqnarray*}
    a sum of acceptable terms. An identical computation shows  
    $$
    \beta_{Q}(q_{*}, (t_{0} +_{Q} t_{1}) -_{Q} t') = \beta_{Q}(q_{*}, (t_{1}-_{Q} t')) + \beta_{Q}(q_{*},t') + \beta_{V} ((t_{1}-_{Q}t'),t')
    $$
    which is a sum of acceptable terms, so $\beta_{Q}(q_{*},t-_{Q}t')$ is acceptable.

    Next, we observe 
    \begin{eqnarray*}
        \beta_{V}((t_{0}+_{Q} t_{1}) -_{Q} q_{*}, t') &=& \beta_{V}((t_{1}-_{Q} q_{*})+t_{0}, t') \\
        &=& \beta_{V}(t_{1}-_{Q} q_{*}, t') + \beta_{V}(t_{0},t'),
    \end{eqnarray*}
    which is a sum of acceptable terms. The case of $\beta_{V}(t-_{Q} q_{*}, t'-_{Q} q_{*})$ is identical.  This proves the lemma. 
\end{proof}

\begin{lem} \label{lem:acceptable}
    Every term in $L_{Q,q_{*}}$ is equal modulo $T_{Q,q_{*}}$ to an acceptable one. 
\end{lem}

\begin{proof}
    Clearly constants and variables are acceptable. Moreover, the acceptable terms are obviously closed under the term-forming operations of $L_{-,q_{*}}$.  Hence it suffices to show that if $t,t'$ are acceptable, then $\beta_{Q}(t,t')$ is acceptable (in the case that $t$ is $Q$-valued and $t'$ is $V$-valued) and $\beta_{V}(t,t')$ is acceptable (when $t$ and $t'$ are $V$-valued).  

    For the first case, we note that 
  $$
  \beta_{Q}(t,t') = \beta_{Q}(q_{*},t') + (\beta_{V}(t-_{Q} q_{*},t'))^{2}
  $$
  which is acceptable, since both $\beta_{Q}(q_{*},t)$ and $\beta_{V}(t-_{Q}q_{*},t')$ are acceptable by Lemma \ref{lem:acceptablelem}.  The second case is even easier, since this follows directly by \ref{lem:acceptablelem}. 
\end{proof}

\begin{thm} \label{thm:nfop2proof}
    The theory $T_{Q}$ is NFOP$_{2}$. 
\end{thm}

\begin{proof}
    Since having FOP$_{2}$ is clearly preserved under adding constants, it suffices to show that $T_{Q,q_{*}}$ is NFOP$_{2}$. By elimination of quantifiers, we know that each formula $\varphi(x,y,z)$ of $L_{Q,q_{*}}$ may be expressed, modulo $T_{Q,q_{*}}$ as $\psi(t_{1}(x,y,z), \ldots, t_{k}(x,y,z))$ for some quantifier-free $L_{-,q_{*}}$ formula $\psi(w_{1}, \ldots, w_{k})$ and $L_{Q,q_{*}}$-terms $t_{1}, \ldots, t_{k}$. Moreover, by Lemma \ref{lem:acceptable} and possibly padding with dummy variables, each term $t_{i}(x,y,z)$ may be expressed as $s_{i}(b_{i,1}(u_{i,1}), \ldots, b_{i,l_{i}}(u_{i,l_{i}}))$, where $s_{i}$ is an $L_{-,q_{*}}$ term, each $b_{i,j}$ is in $\mathcal{B}$, and $u_{i,j}$ is an ordered pair of tuples of variables from $x$, $y$, and $z$. As $T_{-,q_{*}}$ is stable, Fact \ref{fact:compositionlemma} immediately implies then that $\varphi(x,y,z)$ is NFOP$_{2}$. As $\varphi(x,y,z)$ was arbitrary, this completes the proof. 
\end{proof}

\begin{cor}
    The theory $T_{O}$ is NFOP$_{2}$. 
\end{cor}

\begin{proof}
    The theory $T_{O}$ is interpretable (even definable) in $T_{Q,q_{*}}$: 
 if $M \models T_{Q, q_{*}}$, then a model of $T_{O}$ is definable in $M$ by interpreting $q$ as the function $\beta_{Q}(q_{*},\cdot)$ on $V(M)$. Hence $T_{O}$ is NFOP$_{2}$ since $T_{Q,q_{*}}$ is, by Theorem \ref{thm:nfop2proof}. 
\end{proof}

\begin{rem}
    Using unpublished results of Chernikov and Hempel \cite[Theorem 8.11]{ChernikovHempel3}, the same proof gives the even stronger conclusion that $T_{Q}$ and $T_{O}$ are NOP$_{2}$ in the sense of Takeuchi.   
\end{rem}

\subsection{NSOP$_{1}$}

For this subsection, we will let $T$ be a theory of a quadratic geometry, in which the associated field is algebraically closed (the generic case without the Witt defect in the language), perfect PAC with Galois group $\mathbb{Z}_{2}$ (the generic case in the presence of the Witt defect), or pseudo-finite (either with or without the Witt defect).  Let $T^{0}$ denote the theory of quadratic geometries that are substructures of a model of $T$. We will explicitly describe proofs in these four cases, though our arguments certainly would go through with basically no change in the case that the underlying field was only assumed to be NSOP$_{1}$, defining independence in the field in terms of Kim-independence. 

We will let $\mathbb{M} \models T$ be a monster model. We will show that $T$ is NSOP$_{1}$ through an analysis of independence in $T$.

\begin{defn}
    Suppose $M \models T$. 
         We say that a formula $\varphi(x;a)$ \emph{Kim-divides over} $M$ if there is a Morley sequence $(a_{i})_{i < \omega}$ over $M$ in a global $M$-invariant type with $a_{0} = a$ and $\{\varphi(x;a_{i}) : i < \omega\}$ is inconsistent. We write $a \ind^{Kd}_{M} b$ to indicate that $\mathrm{tp}(a/Mb)$ contains no formula that Kim-divides over $M$.
\end{defn}

We will use the following characterization of NSOP$_{1}$ theories in terms of Kim-independence in order to prove that $T$ is NSOP$_{1}$. Although NSOP$_{1}$ is defined by a combinatorial configuration of instances of a formula, the reader can take the following fact as a definition of NSOP$_{1}$:

\begin{fact} \label{fact: symmetry}
    A theory $T$ is NSOP$_{1}$ if and only if $\ind^{Kd}$ is symmetric over models, i.e. for any $M \models T$, $a \ind^{Kd}_{M} b$ if and only if $b \ind^{Kd}_{M} a$. \cite[Theorem 5.16]{KR20}
\end{fact}

It is known that ACF$_{2}$, PSF$_{2}$, and, more generally, the theory of any bounded perfect PAC field is simple, hence NSOP$_{1}$ \cite{hrushovski1991pseudo}.  In particular, the theory of a perfect PAC field with Galois group is $\mathbb{Z}_{2}$ is NSOP$_{1}$. 

\begin{defn}
    Suppose $M \models T$, and $A$ and $B$ are sets of parameters.  We define $A \ind^{*}_{M} B$ if and only if $K(\mathrm{acl}(AM)) \ind^{Kd}_{K(M)} K(\mathrm{acl}(BM))$ and $\langle V(\mathrm{acl}(AM)) \rangle \cap \langle V(\mathrm{acl}(BM)) \rangle = \langle V(M) \rangle$.  
\end{defn}

\begin{rem}
   Note that, in an NSOP$_{1}$ theory, a formula Kim-forks over a model if and only if it Kim-divides over that model \cite[Proposition 3.19]{KR20}, so, with respect to the NSOP$_{1}$ theory of the field, we will not need to distinguish between $\ind^{Kd}$ and $\ind^{K}$. Moreover, it may be surprising at first to see that we do not impose some kind of independence in the $Q$ sort.  But, since we are working over a model, our base contains an element of $Q$ and therefore there is a definable bijection from $V$ to $Q$. Thus, we get an independence condition on $Q$ simply by imposing one on $V$. 
\end{rem} 

In \cite[Lemma 4.7]{kruckman2023new}, the two-sorted theory of vector spaces with a symmetric positive definite bilinear form over a real-closed field was considered and the following two facts were proved.  The proofs go through in our context without change:

\begin{fact} \label{fact:Kimvecfacts}
    Suppose $M \models T$, $A = \mathrm{acl}(AM)$, and $B = \mathrm{acl}(BM)$.  
    \begin{enumerate}
        \item If $A \ind^{Kd}_{M} B$, then $K(A) \ind^{Kd}_{K(M)} K(B)$ in $\mathrm{Th}(K(\mathbb{M}))$. \cite[Lemma 4.7]{kruckman2023new}
        \item If $A \ind^{Kd}_{M} B$, then $V(A)$ and $V(B)$ are linearly independent over $V(M)$. \cite[Lemma 4.8(3)]{kruckman2023new}
    \end{enumerate}
\end{fact}

\begin{prop} \label{prop: indepchar}
    Suppose $M \models T$, $A = \mathrm{acl}(AM)$, and $B = \mathrm{acl}(BM)$.  Then $A \ind^{Kd}_{M} B$ if and only if $A \ind^{*}_{M} B$.  
\end{prop}

\begin{proof}
    Note that if $A \ind^{Kd}_{M} B$ then $A \ind^{*}_{M} B$ by Fact \ref{fact:Kimvecfacts}.  For the other direction, we follow the proof of \cite[Theorem 4.9]{kruckman2023new}.  Assume $A \ind^{*}_{M} B$ and let $(B_{i})_{i < \omega}$ be an $M$-invariant Morley sequence over $M$ with $B_{0} = B$.  Since $K(A) \ind^{Kd}_{K(M)} K(B)$, there is some $K_{*}$ such that $K_{*}K(B_{i}) \equiv_{K(M)} K(A)K(B)$ in $\mathrm{Th}(K(\mathbb{M}))$ for all $i < \omega$. By Corollary \ref{cor:quadraticstableembeddedness}, we get $K_{*}B_{i} \equiv_{M} K(A)B$ for all $i$. We put $\tilde{K} = K(\mathrm{acl}(K_{*}(B_{i})_{i < \omega}))$.  

    Fix a tuple $\overline{m} = (m_{i})_{i < \gamma}$ from $V(M)$ which enumerates a basis.  Then choose $\overline{a} = (a_{i})_{i < \delta}$ from $V(A)$ such that $\overline{m}\overline{a}$ enumerates a basis of $V(A)$ (in $V(A)$ or, equivalently, in $V(\mathbb{M})$).  Likewise, for each $i < \omega$, we fix a tuple $\overline{b}_{i} = (b_{i,j})_{j < \epsilon}$ that $\overline{m}\overline{b}_{i}$ enumerates a basis of $V(B_{i})$. By assumption, the collection of vectors $\overline{m}(\overline{b}_{i})_{i < \omega}$ is linearly independent. Then we let $(\tilde{K},\tilde{V}, \tilde{Q})$ denote the substructure of $\mathbb{M}$ generated by $M$, together with $\tilde{K}$ and $\overline{m}(\overline{b}_{i})_{i < \omega}$.  Then $\tilde{V}$ is the $\tilde{K}$-vector space with basis $\overline{m}(\overline{b}_{i})_{i < \omega}$ and $\tilde{Q}$ is a set of quadratic forms on $\tilde{V}$ with a regular action of $\tilde{V}$. Choose some $q_{0} \in \tilde{Q}$.  Then, for each $v \in \tilde{V}$, we let $q_{v} = v +_{Q} q_{0}$. Equivalently, the quadratic form $q_{v} = q_{0} + \lambda_{v}^{2}$, viewed as a function. Then, by regularity of the action, we have $\tilde{Q} = \{q_{v} : v \in \tilde{V}\}$. 

    Now we will construct an extension $N = (\tilde{K}, W,Q')$ as follows. First, we introduce a new sequence $\overline{a}' = (a'_{i})_{i < \delta}$ (not in $\mathbb{M}$) and let $W$ be the $\tilde{K}$-vector space with basis $\overline{m}\overline{a}'(\overline{b}_{i})_{i < \omega}$.  We interpret the vector space structure on $W$, including the coordinate functions and linear independence predicates, in the natural way.  Next, we let $Q' = \{q_{w} : w \in W\}$ where $q_{w}$ has been defined above for $w \in V$ and, for all $w \in W - V$, $q_{w}$ is a new element, not in $\mathbb{M}$.  To interpret the form on $N$, it suffices to define it on the basis of $W$.  For pairs of basis vectors contained in $\tilde{V}$, we define the form to equal the value of the given form on $\tilde{V}$.  Then, for all $i,i' < \delta$, $j < \gamma$, $k < \epsilon$, and $\ell < \omega$, we define 
    \begin{eqnarray*}
    \beta_{V}^{N}(a'_{i},a'_{i'}) &=& \beta_{V}^{\mathbb{M}}(a_{i},a_{i'}) \\
    \beta_{V}^{N}(a'_{i},m_{j}) &=& \beta_{V}^{\mathbb{M}}(a_{i},m_{j}) \\
    \beta_{V}^{N}(a'_{i},b_{\ell,k}) &=& \beta_{V}^{\mathbb{M}}(a_{i},b_{0,k}).
    \end{eqnarray*}
    Extending $\tilde{K}$-linearly, then, defines a bilinear form $\beta_{V}^{N} : W\times W \to \tilde{K}$.  

    To complete the construction of $N$, we have to interpret the action $+_{Q}$ of $W$ on $Q'$, the `evaluation function' $\beta_{Q}$ on $Q' \times W$, and possibly the Witt index, if present in the language of $T$.  To define the action, we just set, for all $v,w \in W$, $v +_{Q}^{N} q_{w} = q_{v+w}$, which is clearly regular and extends the action of $\tilde{V}$ on $\tilde{Q}$.  This, then, forces the interpretation $q_{v} -_{Q}^{N} q_{w} = v-w$. To define the interpretation of the evaluation function, we will first define $\beta_{Q}^{N}(q_{0},-)$ as a quadratic form on a basis of $W$.  For $i < \gamma$, $\ell < \omega$, and $j < \epsilon$, we define 
    \begin{eqnarray*}
        \beta^{N}_{Q}(q_{0},m_{i}) &=& \beta^{\mathbb{M}}_{Q}(q_{0},m_{i}) \\
        \beta^{N}(q_{0},b_{\ell,j}) &=& \beta^{\mathbb{M}}_{Q}(q_{0},b_{0,j}).
    \end{eqnarray*}
    This ensures that $\beta^{N}_{Q}(q_{0},-)$ agrees with $\beta^{\mathbb{M}}_{Q}(q_{0},-)$ on $\tilde{V}$.  Then, for each $i < \delta$, we set 
    $$
    \beta^{N}_{Q}(q_{0},a'_{i}) = \beta^{\mathbb{M}}_{Q}(q_{0},a_{i}).
    $$
    This function on a basis of $W$ extends uniquely to a quadratic form $\beta^{N}_{Q}(q_{0},-)$ on all of $W$.  Next, given any $v,w \in W$, we define 
    $$
    \beta^{N}_{Q}(q_{w},v) = \beta^{N}_{Q}(q_{0},v) + (\beta^{N}_{V}(w,v))^{2}.
    $$
    This asserts that, as a function, $q_{w} = q_{0} + \lambda^{2}_{w}$, as desired. By unraveling definitions, we see that, for $q \in \tilde{Q}$ and $v \in \tilde{V}$, we have $\beta^{N}_{Q}(q,v) = \beta^{\mathbb{M}}(q,v)$. Finally, if the Witt defect is present in the language of $T$, we need to interpret the Witt defect $\omega^{N} : Q' \to \{0,1\}$ on $N$. Suppose that, in $\mathbb{M}$, $\omega(q_{0}) = i_{*} \in \{0,1\}$. Then we set, for all $w \in W$, 
    $$
    \omega^{N}(q_{w}) = \left\{
        \begin{matrix}
            i_{*} & \text{ if } q_{0}(w) \in \wp(\tilde{K}) \\
            1 - i_{*} & \text{ if } q_{0}(w) \not\in \wp(\tilde{K}). 
        \end{matrix}
    \right.
    $$
    This completes the construction of $N$.  By quantifier elimination, we may embed $N$ into $\mathbb{M}$ over the substructure $(\tilde{K}, \tilde{V}, \tilde{Q})$.  We identify $N$ with its image in $\mathbb{M}$.  

    Let $A'$ be the substructure of $N$ generated by $M$, $K_{*}$, and $\overline{a}'$.  By construction and quantifier elimination, we have $A'B_{i} \equiv_{M} AB$ for all $i < \omega$. As $(B_{i})_{i < \omega}$ was chosen to be an arbitrary $M$-invariant Morley sequence over $M$ starting with $B$, this shows $A \ind^{Kd}_{M} B$.
\end{proof}

\begin{cor} \label{cor:NSOP1}
    The theory $T$ is NSOP$_{1}$.
\end{cor}

\begin{proof}
    Proposition \ref{prop: indepchar} shows in particular that Kim-dividing is symmetric over models in $T$, hence $T$ is NSOP$_{1}$ by Fact \ref{fact: symmetry}.  
\end{proof}

\begin{cor}
    Suppose $T'$ is either $T_{O}$ or a pseudo-finite completion of $T^{0}_{O}$.  Then $T'$ is NSOP$_{1}$. 
\end{cor}

\begin{proof}
    Note that each possibility for $T'$ is interpretable in one of the corresponding possibilities for $T$, hence $T'$ is NSOP$_{1}$, by Corollary \ref{cor:NSOP1}.
\end{proof}

Finally, we also deduce NSOP$_{1}$ for the Granger example over pseudo-finite fields in characteristic not $2$:
\begin{cor}
    The theories $_{S}T^{PSF_{\neq 2}}_{\infty}$ and $_{A}T^{PSF_{\neq 2}}_{\infty}$ are NSOP$_{1}$. 
\end{cor}

\begin{proof}
    The proof is identical to the proof for Proposition \ref{prop: indepchar}, simply skipping the part of the proof that involved the $Q$ sort and using the quantifier elimination results in Corollaries \ref{cor:grangerqe} and \ref{cor:grangerstableembeddedness}.  
\end{proof}

\bibliographystyle{alpha}
\bibliography{biblio.bib}{}

\end{document}